\documentclass[english,a4paper]{article}

\usepackage{ifdraft}

\usepackage[english]{babel}
\usepackage{geometry}
\usepackage{makeidx}
\usepackage{tabto}
\usepackage{fancyhdr}

\usepackage{marvosym}
\usepackage{amsfonts}
\usepackage{amssymb}
\usepackage{stmaryrd}
\usepackage{braket}
\usepackage{bbold}
\usepackage{bbm}

\usepackage{braket}
\usepackage{extarrows}
\usepackage{colonequals}

\usepackage{mathtools}
\usepackage{amsmath}
\usepackage{amsthm}
\usepackage{cases}
\usepackage{multirow, bigdelim}
\usepackage{nicefrac}
\usepackage{relsize}

\usepackage{array}
\usepackage{enumerate}
\usepackage{graphicx}
\usepackage{caption}
\usepackage{subcaption}
\usepackage{tikz}
\usepackage{float}
\usepackage{pdfpages}
\usepackage[nopgf,nographicx,nobookmark]{incgraph}
\usepackage{longtable}

\usepackage{hyperref}

\usepackage{tocbasic}
\usepackage{cite}
\usepackage{url}

\usepackage{xcolor}

\usepackage{marginnote} 
\usepackage[colorinlistoftodos]{todonotes}

\usepackage{textcomp}

\allowdisplaybreaks[1]

\graphicspath{ {images/} }

\usetikzlibrary{fit}
\usetikzlibrary{arrows}
\usetikzlibrary{petri}
\usetikzlibrary{topaths}
\usetikzlibrary{positioning}
\usetikzlibrary{patterns}
\usetikzlibrary{decorations.pathreplacing}
\usetikzlibrary{calc}
\usetikzlibrary{shapes}
\usetikzlibrary{trees}
\usetikzlibrary{intersections}
\tikzset{>=latex'}

\theoremstyle{plain}
\newtheorem{theorem}             {Theorem}[section]
\newtheorem{lemma}      [theorem]{Lemma}
\newtheorem{corollary}  [theorem]{Corollary}

\theoremstyle{definition}
\newtheorem{definition} [theorem]{Definition}

\theoremstyle{definition}
\newtheorem*{lemma*}     {Lemma}
\newtheorem*{claim*}     {Claim}

\theoremstyle{definition}
\newtheorem*{definition*}{Definition}

\newtheorem*{remark*}    {Remark}

\newcommand{\Nat}{\ensuremath{\mathbb{N}}}
\DeclareMathOperator{\dist}{dist}
\DeclareMathOperator{\Aut}{Aut}
\newcommand{\abs}[1]{\left| #1 \right|}

\newcommand{\restr}[2]{
	\ensuremath{
		\left.\kern-\nulldelimiterspace
		#1
		\vphantom{\big|}
		\right|_{#2}
	}
}

\title{Classification of Finite Highly Regular \\Vertex-Coloured Graphs\footnote{The research leading to these results has received funding from the
		European Research Council (ERC) under the European Union’s Horizon 2020
		research and innovation programme (EngageS: grant agreement No.\ 820148).}}
\author{Irene Heinrich, Thomas Schneider, and Pascal Schweitzer\\
TU Kaiserslautern}

\definecolor{darkgreen}{rgb}{0,0.6,0}
\definecolor{darkred}{RGB}{255, 0, 102}
\definecolor{darkblue}{RGB}{51, 204, 204}
\definecolor{darkyellow}{RGB}{155,135,12}
\definecolor{fuchsia}{RGB}{255,0,255}

\hypersetup{breaklinks=true}
\definecolor{hrefblue}{rgb}{0.5,0.5,1.0}
\definecolor{hrefred}{rgb}{0.5,0,0}
\definecolor{hrefgreen}{rgb}{0,0.5,0}
\definecolor{hrefblue}{rgb}{0,0,0.5}

\hypersetup{colorlinks,linkcolor=hrefred,filecolor=hrefgreen,urlcolor=hrefred,citecolor=hrefblue}

\def\env@matrix{\hskip -\arraycolsep
	\let\@ifnextchar\new@ifnextchar
	\array{*\c@MaxMatrixCols c}}

\makeatletter
\renewcommand*\env@matrix[1][c]{\hskip -\arraycolsep
	\let\@ifnextchar\new@ifnextchar
	\array{*\c@MaxMatrixCols #1}}
\makeatother

\DeclareMathOperator{\sylvester}{Slv}

\DeclareMathOperator{\extendedHadamardGraph}{XHdG}

\newcommand{\tinyqed}{\mathbin{\mathsmaller\blacksquare}}

\newcommand{\highlyRegular}[2]{%
	$#1$-HR\textsuperscript{#2}%
}

\newcommand{\highlyRegularnok}{HR}

\newcommand{\tupleRegular}[1]{%
	$#1$-\text{TR}%
}

\newcommand{\ultraRegular}[1]{%
	$#1$-\text{UH}%
}

\newcommand{\infUltraRegular}{
 \text{UH}
}

\newcommand{\firstRedNbGraph}[2]{
	\ensuremath{
		{#1}\left[N^R(#2) \right]
	}
}

\newcommand{\secondRedNbGraph}[2]{
	\ensuremath{
			{#1}\left[R \setminus N^R(#2) \right]
	}
}
\newcommand{\reverseFunction}[1]{
	\ensuremath{
		#1^{-1}
	}
}
\newcommand{\colouring}{
	\ensuremath{
		\chi
	}
}

\newcommand{\redColour}{
	\ensuremath{\text{red}}
}
\newcommand{\blueColour}{
	\ensuremath{\text{blue}}
}

\newcommand{\hadTwelve}{\ensuremath{\mathrm{H}^{12}}}

\newcommand{\disjointUnionCliques}{\ensuremath{sK_{t}}}

\newcommand{\pentagon}{\ensuremath{C_5}}

\newcommand{\rookGraph}[1]{\ensuremath{R_{#1}}}

\newcommand{\comment}[2][]{%
	{%
		\ifthenelse{\equal{#1}{thomas}}%
		{
			\todo[color={blue!50!green!33},size=\tiny]{
			\textbf{T:}~#2}%
		}%
		{}%
		\ifthenelse{\equal{#1}{irene}}%
		{
			\todo[color={yellow!50!red!33},size=\tiny]{
			\textbf{I:}~#2}%
		}%
		{}%
		\ifthenelse{\equal{#1}{pascal}}%
		{
			\todo[color={blue!50!red!33},size=\tiny]{
			\textbf{P:}~#2}%
		}%
		{}%
	}%
}
\newcommand{\icomment}[2][]{%
	{%
		\ifthenelse{\equal{#1}{thomas}}%
		{
			\todo[inline,color={blue!50!green!33}]{
			\textbf{T:}~#2}%
		}%
		{}%
		\ifthenelse{\equal{#1}{irene}}%
		{
			\todo[inline,color={yellow!50!red!33}]{
			\textbf{I:}~#2}%
		}%
		{}%
		\ifthenelse{\equal{#1}{pascal}}%
		{
			\todo[inline,color={blue!50!red!33}]{
			\textbf{P:}~#2}%
		}%
		{}%
	}%
}

\makeatletter
\newcommand{\pWS}
		   {%
            \@ifnextchar.{}{%
           	  \@ifnextchar,{}{%
           	  	\@ifnextchar!{}{%
           	  	  \@ifnextchar?{}{%
           	  	  	\@ifnextchar:{}{%
           	  	  	  \@ifnextchar;{}{%
           	  	  	  	\@ifnextchar-{}{%
           	  	  	  	  \@ifnextchar){}{%
	       	  	  	  	  	\@ifnextchar'{}{%
							  \@ifnextchar"{}{
							  }
           	  	  	  	  	}
           	  	  	  	  }
           	  	  	  	}
           	  	  	  }
           	  	  	}
           	  	  }
           	  	}
           	  }
            }
           }

\tikzstyle{vertex}=[circle,draw,minimum size=.4mm]
\tikzstyle{empty}=[]
\tikzstyle{edge}=[draw,thick]

\newcommand{\footnotenomark}[2]{%
    \addtocounter{footnote}{1}%
    \footnotetext[\thefootnote]{%
        \addtocounter{footnote}{-1}%
        \refstepcounter{footnote}\label{#1}%
        #2%
    }%
}

\newcommand{\footref}[1]{%
    $^{\ref{#1}}$%
} 
 
\begin{document}

    \maketitle
        
\begin{abstract}
A coloured graph is~$k$-ultrahomogeneous if every isomorphism between two induced subgraphs of order at most~$k$ extends to an automorphism. A coloured graph is~$k$-tuple regular if the number of vertices adjacent to every vertex in a set~$S$ of order at most~$k$ depends only on the isomorphism type of the subgraph induced by~$S$.

We classify the finite vertex-coloured~$k$-ultrahomogeneous graphs and the finite vertex-coloured~$\ell$-tuple regular graphs for~$k\geq 4$ and~$\ell\geq 5$, respectively.
Our theorem in particular classifies finite vertex-coloured ultrahomogeneous graphs, where ultrahomogeneous means the graph is  simultaneously~$k$-ultrahomogeneous for all~$k\in \mathbb{N}$.
\end{abstract}

\section{Introduction}

A structure is \emph{ultrahomogeneous} (UH)\footnote{Some authors use the term homogeneous instead. We use the term ultrahomogeneous to highlight that every and not just some isomorphism between isomorphic substructures must extend to an automorphism.} if every isomorphism between two induced substructures within the structure extends to an automorphism of the entire structure. Ultrahomogeneous structures have been extensively studied in the literature (see related work below), in part because they find important applications in model theory and Ramsey theory in particular due to their construction as 
Fra{\"{\i}}ss{\'{e}} limits (see \cite{MR2800979}). 
In this paper we focus on finite simple graphs.

A graph is~\emph{$k$-ultrahomogeneous} ($k$-UH) if every isomorphism between two induced subgraphs of order at most~$k$ extends to an automorphism. 
Note that a graph is 1-ultrahomogeneous if and only if it is transitive.
Moreover, the 2-ultrahomogeneous graphs are precisely the rank 3 graphs, that is, transitive graphs with an automorphism group that acts transitively on pairs of adjacent vertices and also on pairs of non-adjacent vertices.
Ultrahomogeneity is the same as simultaneous~$k$-ultrahomogeneity for all~$k\in \mathbb{N}$.

Conceptually, all these variants of homogeneity are a form of symmetry of the entire graph. In particular, they constitute a global property of the graph. While this makes the graph truly homogeneous, the property may also be difficult to check algorithmically. In contrast to this, a regularity condition is a local condition that does not necessarily imply global symmetries. 
Following Cameron (see~\cite{cameron:strongly-regular-garphs}), a graph~$G$ is called~\emph{$k$-tuple regular} ($k$-TR)
if for every pair of vertex sets~$S$ and~$S'$ on at most~$k$ vertices that induce isomorphic subgraphs of~$G$, the number of vertices adjacent to every vertex of~$S$ is the same as the number of vertices adjacent to every vertex of~$S'$. 
Note that~$1$-tuple regularity is just ordinary regularity and that~$2$-tuple regularity of a graph precisely means that the graph is strongly regular. Also note that~$k$-regularity is a local property and easily checkable algorithmically.

It is clear from the definition that~$k$-ultrahomogeneity implies~$k$-tuple regularity. Conversely, it is a priori not clear that high regularity should imply any form of ultrahomogeneity. In fact, of course not all graphs that are transitive (i.e.,~1-ultrahomogeneous) are strongly regular (i.e.,~2-tuple regular). However, it turns out that~$(k+1)$-tuple regularity is stronger than~$k$-ultrahomogeneity for~$k\geq 4$.
We are not aware of a direct argument for this phenomenon.
In fact the only available type of argument is via the classifications of the respective classes.

Such classifications are known for~$k$-ultrahomogeneity and~$\ell$-tuple regularity for~$k\geq 2$ and~$\ell\geq 5$, respectively (see Table~\ref{tab:classifiaction}).  Introduced by Sheehan~\cite{Sheehan:SmoothlyEmbeddableSubgraphs}, Gardiner~\cite{Gardiner:HomogeneousGraphs} and independently Gol'fand and Klin~\cite{GolfandKlin1978}  classified all finite ultrahomogeneous simple graphs. They are,
up to complementation, disjoint unions of isomorphic complete graphs, the complements of such graphs, the $3 \times 3$ \emph{rook's graph} and the~5-cycle.
It turns out that~$5$-tuple regular graphs are~$5$-ultrahomogeneous and even~$k$-ultrahomogeneous for all~$k\geq 5$ as shown by Cameron~\cite{cameron:6-transitive-graphs} (and independently but unpublished by Gol'fand). They are therefore ultrahomogeneous. The only known graphs that are~$4$-tuple regular but not~$5$-tuple regular are the Schl\"{a}fi graph, the McLaughlin graph, and their complements, of which only the Schl\"{a}fi graph and its complement are~$4$-ultrahomogeneous. While there are no other~$4$-ultrahomogeneous graphs, it is an open question whether there are other~$4$-tuple regular graphs. Should they exist they must be pseudogeometric graphs (see~\cite{MR3758062} for more information).
 The list of 3-ultrahomogeneous graphs that are not 4-tuple regular contains infinite families and exceptional graphs (see Table~\ref{tab:classifiaction} and Section~\ref{sec:prelims} for descriptions of the graphs). For the  classification of 2-ultrahomogeneous graphs we refer to~\cite{MR818821}.

\begin{table}[h!]\renewcommand{\arraystretch}{1.15}\setlength{\tabcolsep}{4.1pt}
	\centering
	\begin{tabular}{l|c|c|c|c|c}
		&\ultraRegular{3}&\tupleRegular{4}&\ultraRegular{4}&\tupleRegular{5}&\infUltraRegular{}\\
		&\cite{DBLP:journals/jct/CameronM85}&\phantom{[}open\phantom{]}&\cite{buczak1980finite}&\cite{cameron:6-transitive-graphs}{\footref{monochromfootnote}}&\cite{Gardiner:HomogeneousGraphs}{\footref{monochromfootnote}}\\
		\cline{1-6}
		$3\times 3$ rook's graph \rookGraph{3} &\multirow{3}{*}{$\checkmark$}&\multirow{3}{*}{$\checkmark$}&\multirow{3}{*}{$\checkmark$}&\multirow{3}{*}{$\checkmark$}&\multirow{3}{*}{$\checkmark$}\\
		5-cycle $C_5$ &&&&\\
		$\disjointUnionCliques$ $(s,t \geq 1)$ &&&&\\
		\cline{1-6}
		\multirow{2}{*}{Schl\"{a}fi graph} &\multirow{2}{*}{$\checkmark$}&\multirow{2}{*}{$\checkmark$}&\multirow{2}{*}{$\checkmark$}&&\\
		&&&&\\
		\cline{1-6}
		\multirow{2}{*}{McLaughlin graph} &\multirow{2}{*}{$\checkmark$}&\multirow{2}{*}{$\checkmark$}&&&\\
		&&&&\\
		\cline{1-6}
				\multirow{2}{*}{possibly other pseudogeometric graphs} &\multirow{2}{*}{}&\multirow{2}{*}{$\checkmark$}&&&\\	
		&&&&\\
		\cline{1-6}
		Clebsch graph
		&\multirow{5}{*}{$\checkmark$}&&&&\\
		Higman-Sims graph &&&&\\
		$m\times m$ rook's graph~$R_m$ ($m\geq 4$)&&&&\\
		aff.~polar graph $G^\varepsilon(\mathbb{F}_2^{2d})$ ($d\geq 3, \varepsilon\!\in\!\{+,-\}$)&&&&\\
		gen.~quadrangle $G(Q_5^-(q))$ ($q$ prime power)&&&&
	\end{tabular}
	\caption{Classification of finite \ultraRegular{k} graphs with $k \geq 3$ and finite \tupleRegular{\ell} graphs with $\ell \geq 4$ (up to complementation). For the classification of finite \ultraRegular{2} graphs see \cite{MR818821}.}
	\label{tab:classifiaction}
\end{table}
\footnotenomark{monochromfootnote}{The classification of~\tupleRegular{5} was also given by Gol'fand (unpublished, see\cite{cameron:6-transitive-graphs} and~\cite{Klin1988}). The classification of\infUltraRegular{}was also given independently by \cite{GolfandKlin1978}. See also~\cite{Reichard2015, Klin1988}.}

In this paper we are interested in investigating the concepts of~$k$-ultrahomogeneity and~$k$-tuple regularity for finite vertex-coloured graphs. For these, all definitions are precisely the same except that all isomorphisms respect vertex colours. In model theoretic terms this is equivalent to allowing an arbitrary number of unary relations. 

\textbf{Motivation.} While applying the homogeneity and regularity concepts to vertex-coloured graphs may be interesting in its own right, our interest stems from an algorithmic perspective. It can be shown by induction on~$k$ that~$k$-tuple regularity is equivalent to requiring that for each ordered subset~$S$ of~$G$ of at most~$k$ vertices the number of one-vertex extensions of~$S$ of a certain isomorphism type (the type is determined by~$N(v)\cap S$ when~$S$ is extended by~$v$) depends only on the ordered isomorphism type of the induced graph~$G[S]$. From this it can be argued that~$k$-tuple coloured graphs are precisely the graphs for which the~$k$-dimensional Weisfeiler-Leman algorithm stabilizes already on the initial colouring. The Weisfeiler-Leman algorithm, with roots in algebraic graph theory, is a central tool for graph isomorphism testing and automorphism group computations~\cite{KieferThesis}.
In recursive algorithms for the graph isomorphism problem, vertex colourings play an important role. Therefore vertex coloured graphs are the natural graph class to study. The colourings are used to introduce irregularities into the graph and thereby ensure recursive progress~\cite{DBLP:journals/jsc/McKayP14}. For example, in Babai's quasipolynomial time graph isomorphism algorithm this is the case~\cite{DBLP:conf/stoc/Babai16}. The algorithm exploits a local to global approach that is based on the observation that when a graph's regularity is sufficiently high, it will in fact be highly symmetric. This can then be exploited in the algorithm.
For a strengthening of such local to global approaches, 
it is imperative to gain a better understanding under which conditions high regularity implies high symmetry in vertex coloured graphs.

\textbf{Our results.}
In this paper we provide a classification for the~$k$-ultrahomogeneous and the~$\ell$-tuple regular finite vertex coloured graphs for~$k\geq 4$ and~$\ell\geq 5$, respectively. For~$\ell=4$, the proofs also give a classification of~$4$-tuple regular coloured graphs in case there are no other monochromatic~$4$-tuple regular graphs than the Schl\"{a}fi graph, the McLaughlin graph, and their complements.

For the classification result we exploit that there are four simple but generic operations 
that preserve the degree of regularity and homogeneity. Specifically these are complementation between two colour classes or within a colour class, colour disjoint union (i.e.,~disjoint union of graphs whose vertices do not share colours), homogeneous matching extension, and homogeneous blow up. In the homogeneous matching extension, a colour class that forms an independent set is duplicated. The duplicate vertices are all coloured with the same previously unused colour. The duplicate vertices are connected to their originals via a perfect matching. In a homogeneous blow-up, all vertices in a colour class that forms an independent set are replaced by cliques of the same size.
A graph is reduced if it cannot be built from graphs of smaller order using these operations.

Our theorem shows that for~$k\geq 5$ all~$k$-tuple regular graphs are obtained from monochromatic and reduced bichromatic~$k$-tuple regular graphs. The monochromatic graphs are shown in Table~\ref{tab:classifiaction}. For the bichromatic case we prove that the only irreducible graph that is~$k$-tuple regular for some~$k\geq 5$~is the extended Hadamard graph corresponding to the~$2\times 2$ Hadamard Matrix. This  matrix is unique up to equivalence and a Sylvester matrix (see Theorem~\ref{thm:4:reg:irreducible:bichrom} and Table~\ref{tab:classifiaction;bichrom}). 
The graph is ultrahomogeneous and thus~$k$-tuple regular for all~$k$.
The underlying uncoloured graph is isomorphic to the Wagner graph (i.e., the M\"{o}bius ladder on 8 vertices).

Our theorem in particular provides the first classification  of finite vertex-coloured ultrahomogeneous graphs.

\textbf{Organization of the paper.} After proving pointers to related literature (next subsection) and preliminaries (Section~\ref{sec:prelims}) we give a formal definition of the regularity concepts considered in the paper (Section~\ref{sec:homogen:tup:reg:and:regpreserving}). We investigate bichromatic highly regular graphs (Section~\ref{sec:finite-graphs-2colours}) before arguing that in the trichromatic case under certain conditions one of the colour classes must be trivially connected (Section~\ref{sec:finite-graphs-multiple-colours}). We finally assemble the statements to our classification theorems for the multicoloured case (Section~\ref{sec:classfication:thms}), leading to our discussion of future work (Section~\ref{sec:future:work}).

\subsection{Related work}

Regarding high regularity of graphs, the classifications for~$k$-tuple regularity and~$k$-homogeneity are summarized in Table~\ref{tab:classifiaction}. We should remark that some authors use the terms~$k$-isoregular in place of~$k$-tuple regularity. There is also a related concept called the~$t$-vertex condition (see~\cite{DBLP:journals/jct/Reichard00}).

While the literature on~$k$-ultrahomogeneity and~$k$-tuple regularity of graphs is limited, there are numerous publications surrounding ultrahomogeneity. Extensive work has been done on directed graphs, for finite and infinite countable graphs.  
On top of that there are also various 
generalization of the homogeneity concept and classification results towards other structures as well as modified requirements on what precisely should be extendable to automorphisms.   For a general survey with a focus 
on infinite graphs we refer to \cite{MR2800979}.

\textbf{Previous work on coloured graphs.} There are several publications regarding ultrahomogeneity of edge-coloured graphs. 
Lockett and Truss~\cite{truss:homogeneous-coloured-multipartite-graphs} have classified multipartite ultrahomogeneous graphs for which each colour class forms an independent set or a clique.
They also allow multiple edge colours between different parts. Most of the work treats the infinite parts.
Prior to that some special cases had been regarded. In~\cite{DBLP:journals/dm/GoldsternGK96} infinite bipartite graphs are considered and in~\cite{truss:countable-homogeneous-multipartite-graphs} multipartite graphs without edge colourings are studied. In~\cite{RoseThesis} some infinite graphs with two vertex colours are considered, but a finite number of edge colours is allowed for edges with distinctly coloured end vertices.

\textbf{Restricted homogeneity.} Researchers have also considered various variants of homogeneity where the requirement that isomorphisms between substructures are extendable to automorphisms is only asked for certain substructures. For example distance transitive graphs are of this kind (see \cite{DBLP:journals/ejc/Bon07}). Related concepts more akin to tuple regularity are arc-transitivity and are~$k$-transitivity (see\cite{cameron:6-transitive-graphs}).
There are also classification results for~connected homogeneity, where only isomorphisms between connected substructures need to extend (see \cite{Hamann2017}). When this is restricted to substructures of order~$k$ we get~$k$-connected homogeneous graphs, for which there is a recent classification for various cases described in~\cite{DBLP:journals/jcta/DevillersFPZ20}. We also refer to that paper for further pointers to literature. In fact our presentation of 3-ultrahomogeneous graphs follows the one provided there.

\textbf{Other structures and other morphisms.} Beyond graphs, there are numerous classifications of ultrahomogeneous objects. We refer to work of 
Devillers~\cite{DevillersThesis} and Macpherson's survey \cite{MR2800979}.
Another concept that has been studied is homomorphism-homogeneity, where homomorphisms between substructures are required to extend to endomorphisms of the entire structure.
More variants are obtained by considering epimorphisms and monomorphisms. We refer to~\cite{MR2195577,LockettThesis, DBLP:journals/jgt/RusinovS10,  DBLP:journals/dm/LockettT14} for pointers. Combinations of altered concepts are also studied. This leads, for example, to connected-homomorphism-homogeneity~\cite{DBLP:journals/jgt/Lockett15}.

\section{Preliminaries}\label{sec:prelims}
For a set $A$ and a natural number $k \in \mathbb{N}$, we set
$\binom{A}{k} \coloneqq \{B \subseteq A \mid |B| = k\}$.
We denote the identity matrix of rank $r$ by $\mathbb{1}_r$.

\paragraph{Graph basics.}
A (simple) \emph{graph}~$G$ is a pair $(V, E)$ of sets, the \emph{vertices} and \emph{edges}, respectively, with $E \subseteq \binom{V}{2}$ and $V \cap E = \emptyset$. In this paper we are only concerned with finite non-empty graphs, that is, we assume~$V$ is finite and non-empty. 
We refer to the vertex set of $G$ as $V(G)$ and to its edge set as $E(G)$.
We call $G$ \emph{edgeless} if $E(G) = \emptyset$.
An edge $\{u,v\}$ has \emph{end vertices} $u$ and $v$.
We denote the length of a shortest path 
from $u$ to $v$ by $\dist(u,v)$. By~$\overline{G}\coloneqq (V,\binom{V}{2}-E)$ we denote the \emph{complement} of~$G$.

We define the \emph{neighbourhood} of a vertex $u$ in $G$ as $N_G(u) \coloneqq \{v\in V \mid uv \in E\}$.
For $U \subseteq V(G)$, we call the graph $G[U]$ with vertex set $U$ and edge set $E(G)\cap \binom{U}{2}$ the \emph{subgraph induced by $U$.}

Let $G$ and $G'$ be two graphs. A bijective map $\varphi\colon V(G) \to V(G')$ is an \emph{isomorphism} from~$G$ to~$G'$ if for every pair of vertices $u, v \in V(G)$ vertex $u$ is adjacent to $v$ in $G$ if and only if $\varphi(u)$ is adjacent to $\varphi(v)$ in $G'$.
Should an isomorphism exist, then we say that $G$ and $G'$ are \emph{isomorphic} and write $G \cong G'$.
An isomorphism from~$G$ to itself is an \emph{automorphism}.

\paragraph{Connection types of vertex sets.}
Let $R$ and $B$ be disjoint vertex subsets of a graph~$G$.
The sets~$R$ and~$B$ are \emph{homogeneously connected} if either every vertex of~$R$ is adjacent to every vertex of~$B$ or if no edge of $G$ joins a vertex of~$R$ with a vertex of~$B$.
Recall that a \emph{perfect matching} is a subset $M$ of $E(G)$ such that no two edges of $M$ have a common end vertex and every vertex of $V(G)$ is an end vertex of some edge in $M$.
We call $R$ and $B$ \emph{matching-connected} if the edges of $G$ with one end vertex in $R$ and one end vertex in $B$ form a perfect matching of $G[R \cup B]$ or the edges of $\overline{G}$ with one end vertex in $R$ and one end vertex in $B$ form a perfect matching $\overline{G}[R \cup B]$.

\paragraph{Designs.}
A \emph{$t$-$(v,k,\lambda)$-design} is a pair $(P, B)$ where $P$ is a set of \emph{points} with $|P| = v$ and $B \subseteq \binom{P}{k}$ is a multiset of \emph{blocks} with the property that every set of $t$ points is contained in precisely~$\lambda$ blocks.
If $k \in \{0,1,v-1,v\}$, then $(P,B)$ is \emph{degenerate}.
Observe that $(P, B)$ is degenerate if and only if in the corresponding incidence graph the two sets $P$ and $B$ are homogeneously connected or matching-connected.
We frequently use the classic result of Hughes~\cite{Hughes1965} that symmetric~$3$-designs are degenerate. Specifically, \begin{equation}\label{eq: degenerate symm 3-designs}
	\text{if $|B|=v$, then $t \in \{1,2\}$ or $(P,B)$ is degenerate.}
\end{equation}
We refer to~\cite{hughes1988design} as a standard book on design theory.

\paragraph{Regularity and strong regularity.}
Let $d \in \mathbb{N}$.
A graph $G$ of order $n$ is \emph{($d$-)regular} if every vertex of $G$ is of degree $d$.
It is \emph{strongly regular} with parameters $(n, d, \lambda, \mu)$ if additionally every pair of adjacent vertices has $\lambda$ common neighbours and every pair of non-adjacent vertices has $\mu$ common neighbours.
It is well known (see~\cite{Godsil2001}) that a strongly regular graph with parameters $(n, d, \lambda, \mu)$ satisfies
\begin{equation}\label{eq: parameterConditionSRGs}
	d(d-\lambda -1) = (n-d-1)\mu.
\end{equation}
If  both $G$ and its complement are connected, then $G$ is \emph{primitive}.\footnote{Some authors exclude imprimitive graphs in the definition of strong regularity.}
Otherwise, $G$ is \emph{imprimitive}.

\begin{lemma}
	\label{lem: partitionSrgIntoTwoSrgs}
	Let $H$ be a
	strongly regular graph with parameter set $(n, d, \lambda, \mu)$.
	If $(V_1, V_2)$ is a partition of $V(H)$ such that the induced subgraphs $H[V_1]$ and $H[V_2]$ are strongly regular with respective parameter sets $(n_1, d_1, \lambda_1, \mu_1)$ and $(n_2, d_2, \lambda_2, \mu_2)$, then the parameter set of $H[V_1]$ determines the parameter set of $H[V_2]$ as follows:
	\begin{align}
		n_2 &= n - n_1, \label{aln: orderOfSrgPartition}&\\
		d_2 &= d-\frac{(d-d_1) n_1}{n_2},\label{aln: degOfSrgPartitions} \\
		\lambda_2 &= 
		\lambda - \frac{\lambda(d-d_2)}{d_2} + \frac{(\lambda-\lambda_1)d_1n_1}{d_2n_2}~\text{if } d_2 \neq 0,~\text{and}
		\label{aln: lambdaOfSrgPartition} \\
		\mu_2 &= \frac{d_2(d_2-\lambda_2-1)}{(n_2-d_2-1)}~\text{if } d_2 \neq n_2-1. \label{aln: muOfSrgPartition}
	\end{align}
	If $d_2=0$ (respectively $d_2 = n_2-1$), then $\lambda_2$ (respectively $\mu_2$) can be chosen arbitrarily since $H[V_2]$ is empty (respectively complete).
\end{lemma}

\begin{proof}
	Since $(V_1, V_2)$ is a partition of $V(H)$, Equation~\eqref{aln: orderOfSrgPartition} follows.
	
	None of the denominators in Equations~\eqref{aln: degOfSrgPartitions} and \eqref{aln: lambdaOfSrgPartition} is zero since $H$ is primitive and the sets $V_1$ and $V_2$ are non-empty.
	Regarding Equation~\eqref{aln: degOfSrgPartitions} note that both $(d-d_1)n_1$ as well as $(d-d_2)n_2$ count the number of edges  of 
	$H$ that have one end vertex in $V_1$ and one in $V_2$.
	We prove Equation~\eqref{aln: lambdaOfSrgPartition}.
	Double counting the triangles which contain vertices of $V_1$ and $V_2$ yields
	\begin{equation} \label{eq: count triangles}
		\frac{1}{2}(d-d_2)n_2\lambda = \frac{1}{2}n_1d_1(\lambda-\lambda_1) + \frac{1}{2}n_2d_2(\lambda-\lambda_2),
	\end{equation}
	where on the left side, all $(d-d_2)n_2$ edges with one end in $V_1$ and the other end in $V_2$ are considered. Such an edge is contained in $\lambda$ triangles and each desired triangle is counted twice. The summands on the right side correspond to the triangles containing an edge with both ends in $V_1$ or both ends in $V_2$, respectively. Equation~\eqref{aln: lambdaOfSrgPartition} is obtained by rearranging Equation~\eqref{eq: count triangles}.
	Observe that Equation~\eqref{aln: muOfSrgPartition} is a reformulation of Equation~\eqref{eq: parameterConditionSRGs} applied to $H[V_2]$.
\end{proof}

\paragraph{Sporadic strongly regular graphs.}
The \emph{Clebsch graph}, the \emph{Schl\"afli graph},  the \emph{Higman-Sims graph}, and  the \emph{McLaughlin graph} are the unique strongly regular graphs (up to isomorphism) with the parameters $(16,5,0,2)$, $(27, 16, 10, 8)$, $(100, 22, 0, 6)$, and $(275,112,30,56)$, respectively.
See~\cite{Godsil2001} for the uniqueness of the first two graphs, \cite{Gewirtz1969HigmanSimsUnique} and \cite{GoethalsUniqueMcL} for the uniqueness of the Highman-Sims graph and the McLaughlin graph, respectively.

\paragraph{Graph families.}
We denote the \emph{complete graph} of order $t$ by $K_t$ and the \emph{cycle} of order~$t$ by~$C_t$.
The disjoint union of $s$ copies of $K_t$ is denoted by $sK_t$.

Let $m \in \mathbb{N}\setminus \{0\}$.
The $m \times m$ \emph{rook's graph} $\rookGraph{m}$ has vertex set
$\{v_{ij}\mid 1 \leq i,j \leq m  \}$ and edge set $\{v_{ij}v_{i'j'}\mid i=i'~\text{or}~j=j'~\text{and}~(i,j)\neq (i',j')\}$.
For $i \in \{1, \dots, m\}$ we call $\{v_{ij}\mid 1 \leq j \leq m \}$ the \emph{$i$-th row} of $\rookGraph{m}$.
(The \emph{$j$-th column} is defined analogously.)
If a graph $G$ is isomorphic to $\rookGraph{m}$, then we say that $G$ is \emph{an $\rookGraph{m}$}.
The graph  $\rookGraph{m}$ is strongly regular with parameters $(m^2, 2m-2, m-2, 2)$.
Conversely, Shrikhande~\cite{shrikhande1959rook} proved that for every $m\neq 4$, a strongly regular graph with parameters $(m^2, 2m-2, m-2, 2)$ is an $\rookGraph{m}$.
For~$m=4$ there exist exactly two types of strongly regular graphs with parameters $(16, 6, 2, 2)$, the $\rookGraph{4}$ and the Shrikhande graph.
Given a strongly regular graph $G$ with parameter set $(16,6,2,2)$, the following is an easy way to recognize the isomorphism type of~$G$: for $v \in V(G)$, if $G[N_G(v)] \cong 2K_3$, then $G \cong \rookGraph{4}$. Otherwise, $G[N_G(v)] \cong C_6$ and $G$ is a Shrikhande graph.

Suppose $q, d' \in \mathbb{N}_{\geq 1}$ and let $\kappa$ be a quadratic form on $\mathbb{F}_q^{d'}$.
A vector $v\in \mathbb{F}_q^{d'}$ is \emph{singular} if $\kappa(v) = 0$.
A subspace $W \subseteq  \mathbb{F}_q^{d'}$ is \emph{totally singular} if every vector in $W$ is singular. For even dimension~$d'=2d$, with~$d \in \mathbb{N}_\geq 1$, there are up to isometry exactly two types quadratic forms. We say that $(\mathbb{F}_2^{2d}, \kappa)$ \emph{is of type $+$ (respectively of type $-$)} if the maximal totally singular subspaces of $\mathbb{F}_2^{2d}$ are $m$-dimensional (respectively $m-1$-dimensional). 
If $(\mathbb{F}_2^{2d}, \kappa)$ is of type~$\varepsilon \in \{+,-\}$, then the \emph{affine polar graph} $G^\varepsilon(\mathbb{F}_2^{2d})$ has vertex set $\mathbb{F}_2^{2d}$, and two vectors $u, v \in \mathbb{F}_2^{2d}$ are adjacent if and only if $\kappa(u-v) = 0$.

Let $(\mathbb{F}_q^6, \kappa)$ be a quadratic space of type $-$.
The bipartite graph $G(Q_5^-(q))$ has as vertex set all one-dimensional and two-dimensional totally singular subspaces of $\mathbb{F}_q^6$ with respect to $\kappa$, where two subspaces are adjacent whenever one is a subspace of the other. These graphs are examples of generalized quadrangles.

\paragraph{Coloured graphs.} A \emph{colouring} of $G$ is a map $\colouring\colon V(G) \to C$.
The tuple $(G, \colouring)$ is a \emph{coloured graph} with \emph{colours} from $C$.
Note that we allow adjacent vertices to have the same colour.
If  $|\chi(V)| = \ell$, then $G$ is \emph{$\ell$-coloured}.
We call $(G, \colouring)$ \emph{monochromatic}, \emph{bichromatic}, or \emph{trichromatic} if it is 1-coloured, 2-coloured, or 3-coloured, respectively.
Let $(G, \colouring)$ and $(G', \colouring')$ be coloured graphs.
An isomorphism $\varphi\colon V(G) \to V(G')$ is \emph{colour-preserving} if $\colouring'(\varphi(u)) = \colouring(u)$ for each $u \in V(G)$.
For colored graphs, from now on all isomorphisms are required to be color-preserving.
If there exists a (colour-preserving) isomorphism between $V(G)$ and $V(G')$, we say that $(G, \colouring)$ and $(G', \colouring')$ are \emph{isomorphic} and write $(G, \colouring) \cong (G', \colouring')$.
An \emph{automorphism} of $(G, \colouring)$ is a colour-preserving isomorphism from $V(G)$ to $V(G)$.
Suppose $U \subseteq V(G)$.
We call $(G, \colouring)[U] \coloneqq (G[U], \restr{\colouring}{U})$ the \emph{(coloured) subgraph} of $(G, \colouring)$ \emph{induced by $U$}.
The \emph{complement} of a coloured graph $(G, \chi)$ is $(\overline{G}, \chi)$.

\paragraph{Homogeneous blow-ups.}
The homogeneous blow-up operation replaces a colour class forming an independent set by a union of cliques of the same size.
More precisely, let $(H, \chi_H)$ and $(G, \chi_G)$ be coloured graphs.
If there is an independent colour class~$R$ of~$(G, \chi_G)$, then we say that~$(H, \chi_H)$ is a \emph{homogeneous blow-up} of~$(G, \chi_G)$ at $R$ if there exists an integer~$t\geq 2$ so that~$(H, \colouring_H)$ is obtained by replacing each vertex~$r$ of~$R$ in~$G$ by a $t$-clique~$C_r$. The new vertices all obtain the colour in $\colouring(R)$.
Thus the neighbourhood of a new vertex~$r' \in C_r$ in~$H$ of colour~$R$ is~$N_H(r')= N_G(r)\cup (C_r\setminus\{r'\})$.

\paragraph{Hadamard matrices.}
A \emph{Hadamard matrix} is a matrix $H \in \{-1,1\}^{s\times s}$ in which each two distinct rows are orthogonal. Such a matrix has rank~$s$.
If additionally all row sums of $H$ are equal, then $H$ is \emph{regular}.
Two Hadamard matrices  $H_1, H_2 \in \{-1, 1\}^{s \times s}$ are \emph{equivalent} if $H_2$ can be obtained from $H_1$ by a sequence of operations that multiply a subset of the rows and columns with~$-1$ as well as row swaps and columns swaps. We write $H_1 \cong H_2$ to indicate that $H_1$ and $H_2$ are equivalent.
The \emph{Kronecker product} of two matrices $M \in \mathbb{R}^{m\times n}$ and $M' \in \mathbb{R}^{m' \times n'}$ is the matrix $M \otimes M' \in \mathbb{R}^{mm' \times nn'}$ with
$(M \otimes M')_{m'(r-1)+v,n'(s-1)+w} =M_{r,s}M'_{v,w}$.
The \emph{Sylvester matrix} of rank $2$ is \[\sylvester(2)\coloneqq \begin{pmatrix}
	1 & 1\\
	1 & -1
\end{pmatrix} \]
For an integer $s \geq 2$, the \emph{Sylvester matrix} of rank $2^s$ is
$\sylvester(2^s) \coloneqq \sylvester(2) \otimes  \sylvester(2^{s-1})$. Specifically, this means that~$\sylvester(2^s) = \begin{pmatrix}
	\sylvester(2^{s-1}) & \phantom{-}\sylvester(2^{s-1})\\
	\sylvester(2^{s-1}) & -\sylvester(2^{s-1})
\end{pmatrix}$.
Observe that Sylvester matrices are Hadamard matrices.

\paragraph{Extended Hadamard graphs.}
If $H$ is a Hadamard matrix, then the graph $G(H)$ with
\begin{align*}
	V(G(H)) &= \bigcup_{1 \leq i \leq n} \{r_i^+, r_i^-, c_i^+, c_i^-\},\\
	E(G(H)) &= \bigcup_{\substack{1 \leq i,j \leq n\\H_{i,j}=1}} \{r_i^+c_j^+, r_i^-c_j^-\} \cup \bigcup_{\substack{1 \leq i,j \leq n\\H_{i,j}=-1}} \{r_i^+c_j^-, r_i^-c_j^+\}
\end{align*}
is the \emph{Hadamard graph corresponding to $H$}.
It is well-known (cf.~\cite{brouwer1989}) that
\begin{equation} \label{eq: hadGraph amply parameters}
	\text{G(H) is $s$-regular, and $|N_{G(H)}(u) \cap N_{G(H)}(v)| = \nicefrac{s}{2}$ if $\dist(u,v) = 2$}.
\end{equation}
Moreover,
\begin{equation}\label{eq: 3 indep vertices have same nb of nbs}
	|N_{G(H)}(u) \cap N_{G(H)}(v) \cap N_{G(H)}(w)| = \nicefrac{s}{4}~\text{if}~\dist(u,v) = \dist(u,w) = \dist(v,w) = 2.
\end{equation}
We extend the edge set of $G(H)$ by $\bigcup_{1 \leq i \leq n}\{r_i^+r_i^-, c_i^+c_i^-\}$
and equip the resulting graph $G'(H)$ with a 2-colouring $\colouring_H\colon V(G'(H)) \to \{\redColour, \blueColour \}$ such that vertices of the form $r_i^+$ or $r_i^-$ are red and vertices of the form $c_i^+$ or $c_i^-$ are blue.
The coloured graph $\extendedHadamardGraph(H) \coloneqq (G'(H), \colouring_H)$ is the \emph{extended Hadamard graph corresponding to $H$}.
See Figure~\ref{fig: 8cycleWithDiagonals} for illustrations of the extended Hadamard graph corresponding to $\sylvester(2)$.
Let $(F, \colouring_F)$ be a coloured graph.
If there is a bijection $\varphi\colon \colouring_F(V(F)) \to \{\redColour, \blueColour\}$ such that $(F, \varphi\circ \colouring_F)$ is isomorphic to $(G'(H), \colouring_H)$ for some Hadamard matrix $H$, then we say that $(F, \colouring_F)$ is an \emph{extended Hadamard graph}.
It holds that
\begin{equation} \label{eq: equi-matrices-equi-graphs}
	H_1 \cong H_2~\text{if and only if}~\extendedHadamardGraph(H_1) \cong \extendedHadamardGraph(H_2).
\end{equation}

\begin{figure}
	\centering
	\definecolor{darkgreen}{rgb}{0,0.6,0}
	\definecolor{darkred}{RGB}{255, 0, 102}
	\definecolor{darkblue}{RGB}{51, 204, 204}
	\definecolor{darkyellow}{RGB}{155,135,12}
	\definecolor{fuchsia}{RGB}{255,0,255}
	\definecolor{lightgray}{RGB}{211,211,211}
	\tikzstyle{edge}=[draw,thick]
	\tikzstyle{vertex}=[circle,draw,minimum size=.4mm]
	
	\scalebox{0.6}{
		\begin{tikzpicture}[scale=.6]

			\node[vertex, darkblue ,fill= darkblue] (c1p) at (0,0) {};
			\node[vertex, darkblue ,fill= darkblue] (c1m) at (2,0) {};
			\node[vertex, darkblue ,fill= darkblue] (c2p) at (4,0) {};
			\node[vertex, darkblue ,fill= darkblue] (c2m) at (6,0) {};
			\node[vertex, darkred ,fill= darkred] (r1p) at (-2,-2) {};
			\node[vertex, darkred ,fill= darkred] (r1m) at (-2,-4) {};
			\node[vertex, darkred ,fill= darkred] (r2p) at (-2,-6) {};
			\node[vertex, darkred ,fill= darkred] (r2m) at (-2,-8) {};

			\begin{scope}[shift={(0.5,-3.5)}]
				\draw[lightgray, rounded corners, fill=lightgray]  (-0.1,-0.1) rectangle (1.1,1.1);
				\draw[edge] (0,0.75) .. controls (0.25,0.75) and (0.25,0.75) .. (0.25,1);
				\draw[edge] (0,0.25) .. controls (.75,0.25) and (0.75,0.25) .. (0.75,1);
				\node[gray!10!black, fill=white, draw, circle,inner sep=0.4,minimum height=0.4cm] at (1.05,-.05) {1};
			\end{scope}
			
			\begin{scope}[shift={(4.5,-3.5)}]
				\draw[lightgray, rounded corners, fill=lightgray]  (-0.1,-0.1) rectangle (1.1,1.1);
				\draw[edge] (0,0.75) .. controls (0.25,0.75) and (0.25,0.75) .. (0.25,1);
				\draw[edge] (0,0.25) .. controls (.75,0.25) and (0.75,0.25) .. (0.75,1);
				\node[gray!10!black, fill=white, draw, circle,inner sep=0.4,minimum height=0.4cm] at (1.05,-.05) {1};
			\end{scope}
			
			\begin{scope}[shift={(0.5,-7.5)}]
				\draw[lightgray, rounded corners, fill=lightgray]  (-0.1,-0.1) rectangle (1.1,1.1);
				\draw[edge] (0,0.75) .. controls (0.25,0.75) and (0.25,0.75) .. (0.25,1);
				\draw[edge] (0,0.25) .. controls (.75,0.25) and (0.75,0.25) .. (0.75,1);
				\node[gray!10!black, fill=white, draw, circle,inner sep=0.4,minimum height=0.4cm] at (1.05,-.05) {1};
			\end{scope}
			
			\begin{scope}[shift={(4.5,-7.5)}]
				\draw[lightgray, rounded corners, fill=lightgray]  (-0.1,-0.1) rectangle (1.1,1.1);
				\draw[edge] (0,0.75).. controls (0.75,0.75) and (0.75,0.75) .. (0.75,1);
				\draw[edge] (0,0.25) .. controls (.25,0.25) and (0.25,0.25) .. (0.25,1);
				\node[gray!10!black, fill=white, draw, circle,inner sep=0.4,minimum height=0.4cm] at (1.05,-.05) {-1};
			\end{scope}
			
			\foreach \x in {1,2}
			{
				\foreach \y in {r,c}
				{
					\draw[edge] (\y\x p) --  (\y\x m);
				}   
			} 
			\foreach \x/\y in {1/0,2/4}
			{
				\begin{scope}[shift={(0,-\y)}]
					\draw[edge] (r\x p) --  (0.4,-2.75)--  (0.5,-2.75);
					\draw[edge] (r\x m) --  (0.4,-3.25) --  (0.5,-3.25);
					\draw[edge] (r\x p) -- (3,-2) -- (4.4,-2.75) -- (4.5,-2.75);
					\draw[edge] (r\x m) -- (3,-4) -- (4.4,-3.25) -- (4.5,-3.25);
				\end{scope}
			}
			\foreach \x/\y in {1/0,2/4}
			{
				\begin{scope}[shift={(\y,0)}]
					\draw[edge] (c\x p) --  (0.75,-2.4) --  (0.75,-2.5);
					\draw[edge] (c\x m) --  (1.25,-2.4) --  (1.25,-2.5);
					\draw[edge] (c\x p) -- (0,-5)  -- (0.75,-6.4) -- (0.75,-6.5);
					\draw[edge] (c\x m) -- (2,-5) -- (1.25,-6.4) -- (1.25,-6.5);
				\end{scope}
			}
			
	\end{tikzpicture}}\hspace{3cm}
	\scalebox{0.6}{
		\begin{tikzpicture}[scale=1.3]
			\def\krad{2cm}
			\def\angle{360/8}
			\foreach \i in {0,2,4,6}{
				\node[vertex, darkblue ,fill= darkblue] (\i) at (270-\angle/2+\i*\angle:\krad) {};
			}
			\foreach \i in {1,3,5,7}{
				\node[vertex, darkred ,fill= darkred] (\i) at (270-\angle/2+\i*\angle:\krad) {};
			}
			\draw[edge] (0)--(1)--(2)--(3)--(4)--(5)--(6)--(7)--(0)
			(0)--(4) (1)--(5) (2)--(6) (3)--(7);
		\end{tikzpicture}
	}
	
	\caption{Two drawings of the extended Hadamard graph $\extendedHadamardGraph(\sylvester(2))$ corresponding to $\sylvester(2)$.}
	\label{fig: 8cycleWithDiagonals}
\end{figure}

\section{Homogeneity, tuple regularity, and regularity preserving operations}\label{sec:homogen:tup:reg:and:regpreserving}

We formally define the two regularity concepts central for our paper. Both of them capture a form of high regularity in a graph.

\begin{definition}[Ultrahomogeneity and tuple regularity]
	
	A coloured graph $(G, \colouring)$ is $k$-\emph{ultra\-ho\-mo\-ge\-neous}, or \ultraRegular{k}, if every colour-preserving isomorphism between two induced coloured subgraphs  of $(G, \colouring)$ of order at most $k$ extends to a colour-preserving automorphism of $(G, \colouring)$.
	Let $R$ be a colour class of $(G, \colouring)$ and $U \subseteq V(G)$.
	We set
	\begin{align*}
		N^R(U) &\coloneqq N_{(G, \colouring)}^R(U) \coloneqq \bigcap_{u \in U} N_G(u) \cap R\text{, and}\\
		\lambda^R(U) &\coloneqq \lambda_{(G, \colouring)}^R(U) \coloneqq \left| N^R(U)\right|.
	\end{align*}
	The coloured graph  $(G, \colouring)$ is \emph{$k$-tuple regular}, or \tupleRegular{k}, if for every pair of subsets $U, U' \subseteq V(G)$ with $|U| = |U'| \leq t$  and $(G, \colouring)[U] \cong (G, \colouring)[U']$  and for every colour class $R$ it holds that
	$\lambda^R(U) = \lambda^R(U')$.
	A graph is \emph{ultrahomogeneous}, or \emph{UH}, if it is~\ultraRegular{k} for every~$k \in \mathbb{N}$.
\end{definition}

Note that $(G, \colouring)$ is \tupleRegular{k} if it is \ultraRegular{k}.
Also note that a monochromatic graph $(G, \colouring)$ is \tupleRegular{2} if and only if $G$ is strongly regular.

\begin{lemma} \label{lem: safeOperations}
	Let $R$ be a colour class of a graph $(G, \colouring)$.
	\begin{enumerate}[(i)]
		\item \label{itm: induced-graph-has-property-pi}
		If $(G, \colouring)$ is \tupleRegular{k} (respectively \ultraRegular{k}) and $C$ is a union of colour classes of $(G, \colouring)$, then $(G, \colouring)[C]$ is \tupleRegular{k} (respectively \ultraRegular{k}).
		\item \label{itm: properties of colour subconstituents}
		If $(G, \colouring)$ is \tupleRegular{k} for some $k \geq 2$, then $G\left[N^R(b)\right]$ and $G\left[R\setminus N^R(b)\right]$ are \tupleRegular{(k-1)} for each $b \in V(G)\setminus R$.
	\end{enumerate}
\end{lemma}
\begin{proof}
	Claim~\eqref{itm: induced-graph-has-property-pi} for ultrahomogeneity has already been observed in~\cite{truss:countable-homogeneous-multipartite-graphs}.
	It follows immediately from the respective definition of the properties \tupleRegular{k} and \ultraRegular{k}.
	\noindent
	We prove~\eqref{itm: properties of colour subconstituents}.
	Let $b \in B$ and $U, U' \subseteq V\left(\firstRedNbGraph{G}{b}\right)$ with $|U|, |U'| \leq k-1$ and $\firstRedNbGraph{G}{b}[U] \cong \firstRedNbGraph{G}{b}[U']$.
	It holds that
	$G[U\cup \{b\}] \cong G[U' \cup \{b\}]$.
	Since $G$ is \tupleRegular{k},
	\begin{align*}
		\lambda^R_{(G, \colouring)\left[N^R_G(b) \right]}(U) = \lambda^R_{(G, \colouring)}(U \cup \{b\}) =
		\lambda^R_{(G, \colouring)}(U' \cup \{b\}) =
		\lambda^R_{(G, \colouring)\left[N^R_G(b) \right]}(U').
	\end{align*}
	Hence, $\firstRedNbGraph{G}{b}$ is \tupleRegular{(k-1)}.
	The result for $\secondRedNbGraph{G}{b}$ follows analogously.
\end{proof}

\section{Bicoloured graphs}
\label{sec:finite-graphs-2colours}
We adhere to the following convention throughout this section.
Unless stated otherwise, $(G, \colouring)$ is a 2-coloured graph with $\colouring \colon V(G) \to \{\redColour, \blueColour\}$.
We set $R \coloneqq \reverseFunction{\colouring}(\redColour)$ and $B \coloneqq \reverseFunction{\colouring}(\blueColour)$.
A~vertex $v$ of $G$ is \emph{red} (respectively \emph{blue}) if $\colouring(v) = \redColour$ (respectively $\colouring(v) = \blueColour$).

Assume that $(G, \chi)$ is \tupleRegular{3} such that $R$ and $B$ induce \tupleRegular{4} subgraphs.
In this section we prove that $R$ and $B$ are homogeneously connected if one of the colour classes induces a primitive graph.
If both sets induce imprimitive graphs and $R$ and $B$ are not homogeneously connected, then either $(G, \chi)$ can be obtained by a homogeneous-blow-up, or $(G, \chi)$ is an extended Hadamard graph.
We prove that the 2-coloured Wagner graph is the only extended Hadamard graph which is \tupleRegular{4}. Indeed, it is even ultrahomogeneous.

\subsection{Colour classes inducing primitive graphs}
\label{sec: primitive graphs induced by a colour class}

The overall goal of this subsection is to prove the following theorem.

\begin{theorem}\label{thm: 2col primitive graphs main}
	Let $(G, \chi)$ be a \tupleRegular{3} graph.
	If $G[R]$ is primitive and \ultraRegular{4} or it is the McLaughlin graph, then $R$ and $B$ are homogeneously connected.
\end{theorem}

Due to the classification of  \tupleRegular{4} monochromatic primitive graphs (see Table~\ref{tab:classifiaction}), it suffices to prove that $R$ and $B$ are homogeneously connected whenever $G[R]$ or its complement is the $3\times 3$ rook's graph, the 5-cycle, the Schl\"afli graph, or the McLaughlin graph.

\begin{lemma}
	\label{lem:cricle-components}
	If $(G, \colouring)$ is  \tupleRegular{2} with $G[R] \cong \pentagon$, then $R$ and $B$ are homogeneously connected.
\end{lemma}

\begin{proof}
	The $\pentagon$ does not allow for a partition into two non-empty sets such that both sets induce a regular graph.
	It follows with~Lemma~\ref{lem: safeOperations}.\eqref{itm: properties of colour subconstituents} that every $b \in B$ satisfies $N^R(b) \in \{\emptyset, R\}$.
	Since $(G, \colouring)$ is \tupleRegular{2}, we obtain $N^R(b)= N^R(b')$ for all $b, b' \in B$.
\end{proof}

\begin{lemma}
	\label{lem: inducedSubgraphsOfRook}
	A strongly regular induced subgraph of a rook's graph is an $sK_t$ for some $s,t \in \mathbb{N}_{\geq 1}$ or a rook's graph.
\end{lemma}

\begin{proof}
	Let~$m \in \mathbb{N}_{\geq 1}$ and let $A \subseteq V(\rookGraph{m})$ be a set of vertices such that $G[A]$ is strongly regular with parameters $(n,d,\lambda, \mu)$.
	Assume that $G[A]$ is not empty.
	If all neighbours of $v_{i,j}$ in $G[A]$ are in row $i$ (respectively in column $j$), then  $d = \lambda +1$, which implies $G[A] \cong sK_d$ for some $s \in \mathbb{N}_{\geq 1}$.
	Otherwise, there are $i', j' \in \{1, \dots, m\}$ with $i' \neq i$ and $j' \neq j$ 
	such that $v_{i', j}, v_{i, j'} \in A$. 
	Every additional neighbour of $v_{i,j}$ in $G[A]-\{v_{i,j'},v_{i',j}\}$ is either in row $i$ and, hence, a common neighbour of $v_{i,j}$ and $v_{i,j'}$ or in column $j$ and, hence, a common neighbour of $v_{i,j}$ and $v_{i',j}$.
	Consequently,
	\[d = |N_{G[A]}(v_{i,j})\cap N_{G[A]}(v_{i,j'})| + |N_{G[A]}(v_{i,j})\cap N_{G[A]}(v_{i',j})| + 2 =  2\lambda + 2.\]
	Since $G[A]$ is a subgraph of G we have $\mu \in \{0,1,2\}$ and by Equation~\eqref{eq: parameterConditionSRGs} $\mu \neq 0$.
	If $\mu = 2$, then $n = (\lambda + 1)^2$ and $G[A]$ is a $\rookGraph{\lambda+1}$ by Shrikhande's Theorem~\cite{shrikhande1959rook} (note that the graph induced by six distinct neighbours of a vertex in a rook's graph is of girth~3, whereas the neighbourhood of a vertex in the Shrinkhande graph induces a $C_6$).
	If $\mu = 1$, then $v_{i', j'} \notin A$ and, hence, $v_{i', j''} \in A$ for some $j'' \notin \{j, j'\}$.
	Since $v_{i', j''}$ has a common neighbour with $v_{i,j'}$ in $A$, we obtain $v_{i, j''} \in A$. This is a contradiction since $v_{i,j}$ and $v_{i', j''}$ have two common neighbours in~$G[A]$.
\end{proof}

\begin{corollary}
	\label{lem:rook-graph}
	If $(G, \colouring)$ is \tupleRegular{3} and $G[R]$ is an $\rookGraph{m}$ for some $m \geq 3$, then $R$ and $B$ are homogeneously connected.
\end{corollary}

\begin{proof}
	Towards a contradiction, suppose that there exists $b \in B$ with $\emptyset \subsetneq R \cap N_G(b) \subsetneq R$.
	By Part~\eqref{itm: properties of colour subconstituents} of Lemma~\ref{lem: safeOperations} the red vertices can be partitioned into two sets $A$ and $B$ such that both graphs, $G[A]$ and $G[B]$, are \tupleRegular{2}.
	By Lemma~\ref{lem: inducedSubgraphsOfRook}, $G[A]$ and $G[B]$ each are a disjoint union of cliques or a rook's graph.
	However, a rook's graph $R_m$ with $m \geq 3$ cannot be partitioned into two induced graphs of the form $sK_t$ and $s'K_{t'}$, respectively.
\end{proof}

\begin{lemma}
	\label{lem: Sch HS McL do not allow for regular partition}
	If $H$ is a McLaughlin graph or a Schl\"afli graph and $(V_1, V_2)$ is a partition of $V(H)$, then one of the graphs $H[V_1]$ and $H[V_2]$ is not strongly regular.
\end{lemma}

\begin{proof}[Proof sketch.]
	Assume that $H$ is the Schl\"afli graph.
	Suppose that $V(H)$ can be partitioned into two sets $V_1$ and $V_2$ such that both graphs $H[V_1]$ and $H[V_2]$ are strongly regular.
	We denote the respective parameter sets by $(n_1,d_1,\lambda_1,\mu_1)$ and $(n_2,d_2,\lambda_2,\mu_2)$.
	The Schl\"afli graph is of order~27, and hence, we may assume by symmetry that $n_1 \leq 13$.
	
	According to Brouwer's list (see~\cite[Subsection 9.9]{MR2882891} and~\cite{BrouwerParamter:online})
	of feasible parameter sets for strongly regular graphs, there are five parameter sets of primitive strongly regular graphs of order at most~13, namely
	$(5,2,0,1)$, $(9,4,1,2)$, $(10,3,0,1)$, $(10,6,3,4)$, and $(13,6,2,3)$.
	Choosing one of these for the parameters of $H[V_1]$ we can compute the parameters of $H[V_2]$ using Lemma~\ref{lem: partitionSrgIntoTwoSrgs}. 
	However, all possible choices for the parameters of $H[V_1]$ lead to a violation of the integrality conditions for the parameters of $H[V_2]$ (see also Appendix~\ref{sec:appendix}).
	It remains to consider parameter sets $(n_1, d_1, \lambda_1, \mu_1)$ with $n_1 \leq 13$ which correspond to imprimitive graphs.
	Since the clique number of the Schl\"afli graph is~6 and its largest independent set is of cardinality~3, we may restrict ourselves to the following graphs:
	\begin{itemize}
		\item $sK_t$ with $s \leq 3$, $t \leq 6$, and $st \leq 13$,
		\item $\overline{s'K_{t'}}$ with $s'\leq 6$, $t'\leq 3$ and $s't' \leq 13$.
	\end{itemize}
	There are 24 of these imprimitive graphs in total.
	Again, all of possible parameter sets lead to a violation of the integrality conditions via Lemma~\ref{lem: partitionSrgIntoTwoSrgs} (see Appendix~\ref{sec:appendix}).
	Altogether, there is no partition of the Schl\"afli into two strongly regular graphs.
	
	With the same approach, we prove that the McLaughlin graph does not allow for such a partition either.
	We refer to Appendix~\ref{sec:appendix} for the detailed calculations.
\end{proof}

\begin{corollary}
	\label{lem:sporadic-graphs}
	If $(G, \colouring)$ is \tupleRegular{3} and $G[R]$ is a McLaughlin graph or a Schl\"afli graph, then $R$ and $B$ are homogeneously connected.
\end{corollary}

\begin{proof}
	Suppose that there exists some blue vertex $b \in B$ with $\emptyset \subsetneq N^R(b) \subseteq R$.
	By Lemma~\ref{lem: safeOperations}\eqref{itm: properties of colour subconstituents}, both graphs $G[N^R(b)]$ and $G[R \setminus N^R(b)]$ are strongly regular.
	This contradicts Lemma~\ref{lem: Sch HS McL do not allow for regular partition}.
\end{proof}

\begin{proof}[Proof of Theorem~\ref{thm: 2col primitive graphs main}]
	Assume that $G[R]$ is primitive and \tupleRegular{4}.
	According to the classification of \tupleRegular{4} graphs (cf.\ \cite{cameron:strongly-regular-garphs}), either $G[R]$ or its complement is isomorphic to $C_5$, $R_3$, the Schl\"afli graph, or the McLaughlin graph.
	The claim follows by applying Lemma~\ref{lem:cricle-components}, Corollary~\ref{lem:rook-graph}, or Corollary~\ref{lem:sporadic-graphs} to $(G, \chi)$ or its complement.
\end{proof}

\subsection{Colour classes inducing imprimitive graphs}
\label{sec: imprimitive graphs induced by colour classes}

In this subsection, we prove the following theorem.
\begin{theorem}\label{thm: imprimitive graphs two colours main}
	If $(G, \chi)$ is \tupleRegular{3} with $G[R] \cong s_RK_{t_R}$ and $G[B] \cong s_BK_{t_B}$, then
	\begin{enumerate}[(i)]
		\item $R$ and $B$ are homogeneously connected, or
		\item $G[R]$ and $G[B]$ each are edgeless or complete and $R$ and $B$ are matching-connected, or \label{itm: match}
		\item \label{itm: hom blowup} $(G, \chi)$ is a homogeneous blow-up, or
		\item $(G, \chi)$ is an extended Hadamard graph.
	\end{enumerate}
\end{theorem}

We first consider the case that both colour classes are edgeless or complete.
This allows us to interpret the graph $(G, \chi)$ as an incidence graph of a structure. 

\begin{lemma}\label{lem: imprimitive both edgeless or complete}
	Let $(G, \chi)$ be \tupleRegular{3}.
	If each of the graphs $G[R]$ and $G[B]$ is edgeless or complete, then $R$ and $B$ are either homogeneously connected or matching-connected.
\end{lemma}

\begin{proof}
	Consider the incidence structure $\mathcal{I}$ with point set
	$R$ and block multiset~$\{\!\{N^R(b) \mid b \in B\}\!\}$.
	The 3-tuple regularity of $(G, \colouring)$ implies that $\mathcal{I}$ is a symmetric 3-design.
	It follows with~\eqref{eq: degenerate symm 3-designs} that $k \in \{0,1,|R|-1,|R|\}$, where $k$ denotes the size of a block of $\mathcal{I}$.
	If $k \in \{0, |R|\}$, then $R$ and $B$ are homogeneously connected.
	Otherwise $k \in \{1, |R|-1\}$, that is, the classes $R$ and $B$ are matching-connected.
\end{proof}

Suppose $b \in B$.
If $R$ and $B$ are not homogeneously connected, then both graphs $G[N^R(b)]$ and $G[R\setminus N^R(b)]$ exist and are strongly regular
by Lemma~\ref{lem: safeOperations}.\eqref{itm: properties of colour subconstituents}.
It follows that $G[N^R(b)]$ and $G[R\setminus N^R(b)]$ are disjoint unions of cliques.
In particular, either
\begin{align}
	&G[N^R(b)] \cong sK_{t_R}~\text{for some $0 < s < s_R$},~\text{or} \label{eq: t1_eq_t2} \\
	&G[N^R(b)] \cong s_RK_t~\text{for some $0 <t < t_R$}. \label{eq: s1_eq_s2}
\end{align}
Note that the values $s$ and $t$ are independent of the choice of $b$ since $(G, \chi)$ is \tupleRegular{3}.
In the following Lemma, we prove that Theorem~\ref{thm: imprimitive graphs two colours main} holds if we assume~\eqref{eq: t1_eq_t2}.

\begin{lemma}\label{lem: bichrom blowup}
	Let $(G, \colouring)$ be \tupleRegular{3} with $G[R] \cong s_RK_{t_R}$ and $G[B] \cong s_BK_{t_B}$.
	Suppose $b \in B$.
	If $\firstRedNbGraph{G}{b} \cong sK_{t_R}$ for some $0 < s < s_R$, then $(G, \colouring)$ can be obtained by a homogeneous blow-up of the independent red colour class of some graph $(H, \colouring_H)$.
\end{lemma}

\begin{proof}
	Since $(G, \colouring)$ is \tupleRegular{3}, it holds that $\firstRedNbGraph{G}{b'} \cong sK_{t_R}$ for each $b' \in B$.
	For each maximal red clique $C$ of $(G, \colouring)$ choose a vertex $r_C \in C$.
	Set $R_H \coloneqq \{r_C \mid \text{$C$ is a maximal red clique}\}$.
	The graph $(H, \colouring_H) \coloneqq (G, \colouring)[B \cup R_H]$ satisfies the required conditions.
\end{proof}

Finally, we assume~\eqref{eq: s1_eq_s2} and prove that $(G, \chi)$ is an extended Hadamard graph:

\begin{lemma}\label{lem: das Lemma mit den Claims}
	Let $(G, \colouring)$ be \tupleRegular{3} with
	$G[R] \cong s_RK_{t_R}$ and $G[B] \cong s_BK_{t_B}$.
	If $\firstRedNbGraph{G}{b} \cong s_RK_{t}$ for some $b \in B$ and $0 < t <t_R$,
	then $(G, \colouring)$ is an extended Hadamard graph.
\end{lemma}

\begin{proof}
	For $i \in \{1, \dots, 8\}$ let $U_i$ be a subset that induces a graph as depicted in Figure~\ref{fig: types}.
	Since $(G, \colouring)$ is \tupleRegular{3}, the values $\lambda^R(U_i)$ and $\lambda^B(U_i)$ are determined by the isomorphism type of $(G, \colouring)[U_i]$.
	\begin{figure}[h!]
		\centering
		\scalebox{0.8}{
			\begin{tikzpicture}[scale=1]
				\begin{scope}
					\draw
					(0,0) node (r1) [vertex, fill=darkred, color=darkred] {}
					(1,0) node (b1) [vertex, fill=darkblue, color=darkblue] {}
					(.5, -.7) node (lG1) {$(G, \colouring)[U_1]$}
					;
				\end{scope}
				
				\begin{scope}[shift = {(3,0)}]
					\draw
					(0,0) node (r1) [vertex, fill=darkred, color=darkred] {}
					(1,0) node (b1) [vertex, fill=darkblue, color=darkblue] {}
					(.5, -.7) node (lG1) {$(G, \colouring)[U_2]$};
					\draw[thick] (r1)--(b1);
				\end{scope}
				
				\begin{scope}[shift = {(6,0)}]
					\draw
					(0,0) node (r1) [vertex, fill=darkred, color=darkred] {}
					(0,.6) node (r2) [vertex, fill=darkred, color=darkred] {}
					(1,0) node (b1) [vertex, fill=darkblue, color=darkblue] {}
					(.5, -.7) node (lG2) {$(G, \colouring)[U_3]$}
					;
					\draw[thick] (r2)--(r1)--(b1);
				\end{scope}
				
				\begin{scope}[shift = {(9,0)}]
					\draw
					(0,0) node (r1) [vertex, fill=darkred, color=darkred] {}
					(1,0) node (b1) [vertex, fill=darkblue, color=darkblue] {}
					(1,.6) node (b2) [vertex, fill=darkblue, color=darkblue] {}
					(.5, -.7) node (lG2) {$(G, \colouring)[U_4]$};
					\draw[thick] (r1)--(b1)--(b2);
				\end{scope}
				
				\begin{scope}[shift = {(12,0)}]
					\draw
					(0,0) node (r1) [vertex, fill=darkred, color=darkred] {}
					(1,0) node (b1) [vertex, fill=darkblue, color=darkblue] {}
					(1,.6) node (b2) [vertex, fill=darkblue, color=darkblue] {}
					(.5, -.7) node (lG2) {$(G, \colouring)[U_5]$};
					\draw[thick] (b1)--(r1)--(b2);
				\end{scope}
				
				\begin{scope}[shift = {(0,-3)}]
					\draw
					(0,0) node (r1) [vertex, fill=darkred, color=darkred] {}
					(1,0) node (b1) [vertex, fill=darkblue, color=darkblue] {}
					(1,.6) node (b2) [vertex, fill=darkblue, color=darkblue] {}
					(.5, -.7) node (lG2) {$(G, \colouring)[U_6]$};
				\end{scope}
				
				\begin{scope}[shift = {(3,-3)}]
					\draw
					(1,0) node (b1) [vertex, fill=darkblue, color=darkblue] {}
					(1,.6) node (b2) [vertex, fill=darkblue, color=darkblue] {}
					(1,1.2) node () [vertex, fill=darkblue, color=darkblue] {}
					(.5, -.7) node (lG2) {$(G, \colouring)[U_7]$};
				\end{scope}
				
				\begin{scope}[shift = {(6,-3)}]
					\draw
					(1,0) node (b1) [vertex, fill=darkblue, color=darkblue] {}
					(1,.6) node (b2) [vertex, fill=darkblue, color=darkblue] {}
					(.5, -.7) node (lG2) {$(G, \colouring)[U_8]$};
				\end{scope}
				
				\begin{scope}[shift = {(9,-3)}]
					\draw
					(0,0) node (r1) [vertex, fill=darkred, color=darkred] {}
					(1,0) node (b1) [vertex, fill=darkblue, color=darkblue] {}
					(1,.6) node (b2) [vertex, fill=darkblue, color=darkblue] {}
					(.5, -.7) node (lG2) {$(G, \colouring)[U_9]$};
					\draw[thick] (r1)--(b1);
				\end{scope}
				
			\end{tikzpicture}
		}
		\caption{Some induced subgraphs of $(G, \colouring)$.}
		\label{fig: types}
	\end{figure}
	
	\medskip
	
	\noindent \textbf{Claim 1:} 
	For $X \in \{B, R\}$ let $C^X \subseteq X$ be a maximal clique in $G[X]$.
	The design with points $C^B$ and blocks $\{\!\{N(r)\cap C^B \mid r \in C^R\}\!\}$ is a symmetric 2-$(t_R, t, \lambda^R(U_4))$-design.
	In particular,~$t_B = t_R$ and $G[N^B(r)] \cong s_BK_{t}$ for each $r \in R$.
	
	\medskip
	
	\noindent \textit{Proof of Claim 1:}
	Let $b_1, b_2 \in C^B$ and  $r_1, r_2 \in C^R$ be four distinct vertices.
	It holds that
	\begin{align*}
		|N_G(r_1) \cap C^B| &= \lambda^B(U_1) = \lambda^B(U_2) + 1~\text{and}\\
		t = |N_G(b_1) \cap C^R| &= \lambda^R(U_1) = \lambda^R(U_2) + 1.
	\end{align*}
	Observe that $N_G(b_1) \cap C^R \neq N_G(b_2) \cap C^R$.
	(Otherwise all blue vertices adjacent to $b_1$ have the exactly same neighbours in $C^R$ which implies $t \in \{0,t_R\}$).
	Consequently, there exists a red vertex $r' \in C^R$ such that $(G, \colouring)[\{b_1, b_2, r'\}] \cong (G, \colouring)[U_4]$.
	Hence
	\begin{equation}
		|N_G(b_1) \cap N_G(b_2) \cap C^R| = \lambda^R(U_4).
	\end{equation}
	Interchanging the roles of $R$ and $B$, we obtain
	\begin{equation}
		|N_G(r_1) \cap N_G(r_2) \cap C^B| = \lambda^B(U_3).
	\end{equation}
	Altogether, the considered structure is a symmetric 2-$(t_R, t, \lambda^R(U_4))$-design.
	It follows with Fisher's inequality that $t_B = |C^B| = |C^R| = t_R$ and $G[N^B(r)] \cong s_BK_{t}$ for each $r \in R$.
	$\tinyqed$
	
	\medskip
	
	\noindent \textbf{Claim 2:} There does not exist a parameter $\lambda_n$ such that every pair of non-adjacent blue vertices has exactly $\lambda_n$ common neighbours in every maximal red clique.
	
	\medskip
	
	\noindent \textit{Proof of Claim 2:}
	Suppose towards a contradiction that every pair of non-adjacent blue vertices has $\lambda_n$ common neighbours in every maximal red clique.
	Let $C^B_1$ and $C^B_2$ be two distinct maximal blue cliques and let $C^R$ be a maximal red clique.
	Fix some $b_1 \in C^B_1$.
	Double-counting all vertex sets~$\{b_1,b_1',r\}$ that induce a triangle with $b_1' \in C^B_1$, $r \in C^R$ and all vertex sets~$\{b_1,r,b_2\}$ that induce a subgraph isomorphic to $(G, \colouring)[U_5]$ with $r \in C^R$ and $b_2 \in C^B_2$ yields
	\begin{equation} \label{eq: doublecounting triangles and paths}
		t_R\lambda_n = t^2~\text{and}~(t_R-1)\lambda^R(U_4) = t(t-1),
	\end{equation}
	respectively.
	This implies
	\[t = t_R(\lambda_n-\lambda^R(U_4)) + \lambda^R(U_4).\]
	If $\lambda_n > \lambda^R(U_4)$, then $t \geq t_R$ which contradicts the assumption $t < t_R$.
	If $\lambda_n = \lambda^R(U_4)$, then $t = \lambda^R(U_4)$ which implies that the two blue vertices in $U_4$ are false twins.
	This is a contradiction since the red vertex in $U_4$ is adjacent to exactly one of the two blue vertices.
	Finally, if $\lambda_n < \lambda^R(U_4)$, then $t \leq 0$ which contradicts the assumption $t > 1$.
	$\tinyqed$
	
	\medskip
	
	\noindent \textbf{Claim 3:}
	$t = \nicefrac{t_R}{2}$.
	
	\medskip
	\noindent \textit{Proof of Claim 3:}
	If $t < \frac{t_R}{2}$, then for each two non-adjacent blue vertices $b_1$ and $b_2$ and each maximal red clique $C^R$ there exists $r' \in C^R$ such that $(G, \colouring)[\{b_1, b_2, r'\}] \cong (G, \colouring)[U_6]$ and, hence, $|C^R \cap N_G(b_1) \cap N_G(b_2)| = \lambda^R(U_6)$ which contradicts Claim~2.
	If $t > \nicefrac{t_R}{2}$, then for each two non-adjacent blue vertices $b_1$ and $b_2$ and each maximal red clique $C^R$ it holds that $|C^R \cap N_G(b_1) \cap N_G(b_2)| = \lambda^R(U_5) + 1$, which is a contradiction to Claim~2.
	$\tinyqed$
	
	\medskip
	\noindent \textbf{Claim 4:} If $b_1$ and $b_2$ are two non-adjacent blue vertices and $C^R$ is a maximal red clique, then $|C^R \cap N_G(b_1) \cap N_G(b_2)| \in \left\{0, \nicefrac{t_R}{2}\right\}$.
	
	\medskip
	\noindent \textit{Proof of Claim~4:}
	It follows from Claim~1 that $N_G(b_1) \cap C^R \neq \emptyset$.
	For every $r' \in N_G(b_1) \cap C^R$ it holds that
	\begin{equation} \label{eq: claim4}
		|C^R \cap N_G(b_1) \cap N_G(b_2)| =
		\begin{cases}
			\lambda^R(U_5)+1 &\text{if~$r'$ is adjacent to $b_2$,}\\
			\lambda^R(U_9) &\text{otherwise.}	
		\end{cases}
	\end{equation}
	By Claim~2 we may conclude that
	\[\lambda^R(U_9) \neq \lambda^R(U_5)+1.\]
	This implies, that all $r' \in N_G(b_1)\cap C^R$ satisfy the same condition on the right side of~\eqref{eq: claim4}.
	Altogether, if $N_G(b_1) \cap N_G(b_2) \cap C^R \neq \emptyset$, then $N_G(b_1) \cap C^R = N_G(b_2) \cap C^R$. Claim~4 follows since $|N_G(b_1)\cap C^R| = \nicefrac{t_R}{2}$ by Claim~3.
	$\tinyqed$
	
	\medskip
	\noindent \textbf{Claim 5:} $t_R =2$, $t = 1$.
	
	\medskip
	\noindent \textit{Proof of Claim~5:}
	Let $C^R$ be a maximal red clique and $C^B_1$, $C^B_2$ be two distinct maximal blue cliques.
	Further let $b_1 \in C^B_1$.
	By Claims~2 and~4, there exist vertices $b_2$ and $b_2'$ in $C^B_2$ such that
	$C^R \cap N_G(b_1) = C^R \cap N_G(b_2)$ and
	$C^R \cap N_G(b_1) = C^R \setminus N_G(b_2')$.
	It follows that $\lambda^R(U_4) = 0$.
	As argued in the proof of Claim~1 (see Equation~\eqref{eq: doublecounting triangles and paths}), we have
	\[0 = \lambda^R(U_4)(t_R-1) = t(t-1).\]
	Hence, $t = 1$ and it follows with Claim~3 that $t_R=2$.
	$\tinyqed$
	
	\bigskip
	\noindent
	We obtain that $G[R] \cong s_RK_2$ and $G[B] \cong s_BK_2$ and the edges that join a blue 2-clique with a red 2-clique form a perfect matching.
	
	\bigskip
	\noindent
	It remains to show that $(G, \colouring)$ is an extended Hadamard graph.
	Choose an ordering $C_1^R, \dots, C_{s_R}^R$ of the red 2-cliques and an ordering $C_1^B, \dots, C_{s_B}^B$ of the blue 2-cliques of $(G, \colouring)$.
	Further, relabel the vertices of $G$ so that $C_i^R = \{r_i^+, r_i^-\}$ and $C_i^B = \{c_i^+, c_i^-\}$ for each $i \in \{1, \dots, s\}$.
	Consider the matrix $H \in \{-1, 1\}^{s \times s}$ with
	\[H_{i,j} = \begin{cases}
		\phantom{-}1 &\text{if $c_i^+$ is adjacent to $r_j^+$,} \\
		-1 &\text{otherwise}
	\end{cases} \]
	for all $i, j \in \{1, \dots, s\}$.
	Observe that
	\[|N_G(c_i^+) \cap N_G(c_j^+)| = \lambda^R(U_8) = |N_G(c_i^+) \cap N_G(c_j^-)|.\]
	Hence, the rows of $H$ are pairwise orthogonal.
	Interchanging the roles of $R$ and $B$, we obtain that the columns of $H$ are pairwise orthogonal.
	In particular, $s_R = s_B$, matrix $H$ is a Hadamard matrix, and $(G, \colouring)$ is isomorphic to the extended Hadamard graph of $H$.
\end{proof}

\begin{proof}[Proof of Theorem~\ref{thm: imprimitive graphs two colours main}]
	Let $(G, \chi)$ be \tupleRegular{3} with $G[R] \cong s_RK_{t_R}$ and $G[B] \cong s_BK_{t_B}$.
	Assume that $R$ and $B$ are not homogeneously connected.
	If both induced subgraphs $G[R]$ and $G[B]$ are edgeless or complete, then Lemma~\ref{lem: imprimitive both edgeless or complete} implies that $R$ and $B$ are matching-connected.
	Therefore, we may assume that $G[R]$ is neither edgeless nor complete, that is, $s_r, t_R \geq 2$.
	As argued above, either $G[N^R(b)]$ is isomorphic to $sK_{t_R}$ for some $0 < s < s_R$ or to $s_RK_t$ for some $0 < t < t_R$.
	In the first case, Lemma~\ref{lem: bichrom blowup} implies that $(G, \chi)$ is a blow-up.
	In the second case, $(G, \chi)$ is an extended Hadamard graph by Lemma~\ref{lem: das Lemma mit den Claims}.
\end{proof}

\subsection{Extended Hadamard graphs}

In this subsection we first classify the extended Hadamard graphs which are \ultraRegular{3} (Theorem~\ref{thm:high-regularity-of-hadamard-graphs}).
Our strategy for this is as follows. 
We use the classification of {\'O}~Cath\'ain~\cite{cathain2011} of Hadamard graphs whose row permutation groups are 3-transitive. The 3-transitivity of the automorphism group's action on rows (respectively columns) of the Hadamard matrix $H$ implies 3-transitivity of the induced action of~$\Aut(G(H))$ on the red (respectively blue) cliques of $G(H)$.
However, we need to analyse the action on the vertices rather than on the cliques. 

First, we need to show that every triple of independent red vertices can be mapped to every triple of independent red vertices. Due to the 3-transitive action on cliques, for this, we only need to show that within the subgroup of automorphisms fixing the first three cliques as sets, we can arbitrarily permute the vertices within each clique.
It thus suffices to find for each of the first 3 cliques an automorphism that swaps the two vertices within the clique but fixes vertices within the other two cliques. (The action on cliques other than the first three can be arbitrary.)
For Hadamard matrices that are Sylvester matrices this is shown in Lemma~\ref{lem: aut of sylverster} and for the matrix of rank 12 this is shown in Lemma~\ref{lem:automorphism-had12}\eqref{lem:automorphism-had12:3rows}).

Second, we need to show that for every pair of triples, each consisting of two non-adjacent red vertices and a blue vertex, there is an automorphism mapping one triple to the other.
Once we have shown transitivity of the automorphism group's action on pairs of non-adjacent red vertices in the first step, it suffices to show that among the subgroup of automorphisms that fix all points in the first two cliques pointwise, every blue vertex~$b$ can be mapped to all blue vertices that have the same neighbourhood within the first two red cliques as~$b$. For Sylvester matrices this is shown in Lemma~\ref{lem:aut-of-sylvester:2rows1column} and for the Matrix of rank 12 this is shown in Lemma~\ref{lem:automorphism-had12}\eqref{lem:automorphism-had12:2rows1column}.

We present our proof in terms of Hadamard matrices. A square matrix is \emph{monomial} if each row and each column contains exactly one non-zero entry.
We denote the set of all monomial matrices in $\{-1, 1\}^{s \times s}$ by $\mathcal{M}_s$.
Then $\mathcal{M}_s \times \mathcal{M}_s$ acts on the order-$s$ Hadamard matrices via $(A,B)H \coloneqq AHB^{-1}$.
If $H$ is a Hadamard matrix, then the \emph{automorphism group} $\Aut(H)$ of~$H$ is
the stabiliser of $H$ under this action.
An element of $\Aut(H)$ is an \emph{automorphism} of $H$.
Observe that $A \in \mathcal{M}_s$ has a unique factorisation $A = D_AP_A$ into a diagonal matrix $D_A$ and a permutation matrix $P_A$.
For a Hadamard matrix $H$ and $(A, B) \in \Aut(H)$ set $\nu(A,B) = P_A$. 
Set $\mathcal{A}(H) \coloneqq \nu(\Aut(H))$ to be the induced row permutation group.

\begin{theorem}[\hspace{1sp}\cite{cathain2011}]
	\label{thm: 3-transitive-hadamard-matrices}
	If $H \in \{-1,1\}^{s\times s}$ is a Hadamard matrix such that $\mathcal{A}(H)$ is 3-transitive, then $H$ is equivalent to a Sylvester matrix or $s = 12$.
\end{theorem}

By the theorem, we only need to consider graphs which are constructed from Hadamard matrices or from the  (up to equivalence) unique Hadamard matrix of rank~12. By symmetry, these are also precisely the Hadamard whose column permutation groups act 3-transitively.

\begin{lemma}\label{lem: kronecker automorphisms of sylvester matrices}
	If $H \in \{-1, 1\}^{\{s\times s\}}$ is a Hadamard matrix and $(A, B) \in \Aut(H)$, then
	\begin{enumerate}[(i)]
		\item \label{itm: blow up aut dont change anything}
		$\left(\begin{pmatrix}
			1 & 0\\
			0 & 1	
		\end{pmatrix}
		\otimes
		A,
		\begin{pmatrix}
			1 & 0\\
			0 & 1	
		\end{pmatrix}
		\otimes
		B\right) \in \Aut\left(
		\sylvester(2) \otimes H
		\right)$ and
		\item \label{itm: blow up aut change the last columns with first colums}
		$\left(\begin{pmatrix}
			1 & 0\\
			0 & -1	
		\end{pmatrix}
		\otimes
		A,
		\begin{pmatrix}
			0 & 1\\
			1 & 0	
		\end{pmatrix}
		\otimes
		B\right) \in \Aut\left(
		\sylvester(2) \otimes H
		\right)$.
	\end{enumerate}
\end{lemma}

\begin{proof}The statement follows with
	\begin{align*}
		\begin{pmatrix}
			A & 0\\
			0 & A
		\end{pmatrix}
		\begin{pmatrix}
			H & H\\
			H & -H
		\end{pmatrix}
		\begin{pmatrix}
			B & 0\\
			0 & B
		\end{pmatrix}^{-1}
		=
		\begin{pmatrix}
			AHB^{-1} & AHB^{-1}\\
			AHB^{-1} & -AHB^{-1}
		\end{pmatrix}
		=
		\begin{pmatrix}
			H & H\\
			H & -H
		\end{pmatrix}
	\end{align*} and
	\begin{align*}
		\begin{pmatrix}
			A & 0\\
			0 & -A
		\end{pmatrix}
		\begin{pmatrix}
			H & H\\
			H & -H
		\end{pmatrix}
		\begin{pmatrix}
			0 & B\\
			B & 0
		\end{pmatrix}^{-1}
		=
		\begin{pmatrix}
			AHB^{-1} & AHB^{-1}\\
			AHB^{-1} & -AHB^{-1}
		\end{pmatrix}
		=
		\begin{pmatrix}
			H & H\\
			H & -H
		\end{pmatrix}.
	\end{align*}
\end{proof}

\begin{lemma} \label{lem: aut of sylverster}
	Let $H=\sylvester(2^t)$ for some $t \geq 2$.
	For each $i \in \{1,2,3\}$, there exists an automorphism $(A, B) \in \Aut(H)$ such that $A_{i,i} = -1$ and $A_{j,j} = 1$ for all $j \in \{1,2,3\}\setminus \{i\}$.
\end{lemma}

\begin{proof}
	It is easy to see that
	\[ \left(\begin{pmatrix}
		-1 & 0\\
		0 & 1
	\end{pmatrix},
	\begin{pmatrix}
		0 & -1\\
		-1 & 0
	\end{pmatrix}\right)
	,~
	\left(\begin{pmatrix}
		1 & 0\\
		0 & -1
	\end{pmatrix},
	\begin{pmatrix}
		0 & 1\\
		1 & 0
	\end{pmatrix}\right),~\text{and}~
	\left(\mathbb{1}_2,
	\mathbb{1}_2\right).
	\]
	are automorphisms of $\sylvester(2)$.
	We prove the lemma by induction  on $t$.
	If $t=2$, then  by Lemma~\ref{lem: kronecker automorphisms of sylvester matrices}
	\begin{align*}
		&\left(\begin{pmatrix}
			1 & 0\\
			0 & -1	
		\end{pmatrix}
		\otimes
		\begin{pmatrix}
			-1 & 0\\
			0 & 1
		\end{pmatrix},
		\begin{pmatrix}
			0 & 1\\
			1 & 0	
		\end{pmatrix}
		\otimes
		\begin{pmatrix}
			0 & -1\\
			-1 & 0
		\end{pmatrix}\right),\\
		&\left(\mathbb{1}_2
		\otimes
		\begin{pmatrix}
			1 & 0\\
			0 & -1
		\end{pmatrix},
		\mathbb{1}_2
		\otimes
		\begin{pmatrix}
			0 & 1\\
			1 & 0
		\end{pmatrix}\right),~\text{and} \\
		&\left(\begin{pmatrix}
			1 & 0\\
			0 & -1	
		\end{pmatrix}
		\otimes
		\mathbb{1}_2,
		\begin{pmatrix}
			0 & 1\\
			1 & 0	
		\end{pmatrix}
		\otimes
		\mathbb{1}_2\right)
	\end{align*}
	are automorphisms of $\sylvester(4)$ with the respective desired properties.
	Assume that the statement is satisfied for some $t\geq 2$ and for $i \in \{1,2,3\}$ suppose $(A^{(i)}, B^{(i)}) \in \Aut(\sylvester(2^t))$ such that $A^{(i)}_{i,i} = -1$ and $A^{(i)}_{j,j}=1$ for each $j \in \{1,2,3\}\setminus \{i\}$.
	By Lemma~\ref{lem: kronecker automorphisms of sylvester matrices}, we have $\left(\mathbb{1}_2 \otimes A^{(i)}, \mathbb{1}_2 \otimes B^{(i)} \right) \in \Aut(\sylvester(2^{t+1}))$ and
	$(\mathbb{1}_2 \otimes A^{(i)})_{i,i} = -1$ whereas
	$(\mathbb{1}_2 \otimes A^{(i)})_{j,j} = 1$.
\end{proof}

Next, we search automorphisms for a Sylvester matrix that act on some of the columns in a particular manner while fixing the first and the second row.
Additionally we want the coefficients with which the first two rows are multiplied to be 1. 

Given a column $c$ of a Hadamard matrix $H \in \{-1,1\}^{s \times s}$, let~$\omega_c(H)$ be the set of columns whose first two entries agree with the first two entries of~$c$ or which are their inverses.
More formally, we define 
\[	\omega_c(H)= \{c' \in \{1, \dots, s\} \mid (H_{1,c'}, H_{2,c'}) \in \{(H_{1,c}, H_{2,c}), (-H_{1,c}, -H_{2,c}) \}\}.
\]

\begin{lemma}
	\label{lem:aut-of-sylvester:2rows1column}
	Let $t \in \mathbb{N}_{\geq 2}$ and let $c$ be a column of $\sylvester(2^t)$.
	For each $j \in \omega_c(\sylvester(2^t))$, there exists an automorphism $(A,B)\in \Aut(\sylvester(2^t))$
	such that
	$A_{1,1}=A_{2,2} = 1$ and $B_{c,j} \in \{-1,1\}$.
\end{lemma}

\begin{proof}
	We prove the statement by induction on $t$.
	If $t=2$, then $|\omega_c(\sylvester(2^t))|=1$ and the identity automorphism satisfies the claim.
	Assume from now on that the claim holds true for some $t \geq 2$.
	Let $j$ be a column of $\sylvester(2^{t+1})$ with $j \in \omega_c(\sylvester(2^{t+1}))$.
	Observe that $(j \bmod 2^t+1) \in \omega_{(c \bmod 2^t+1)}(\sylvester(2^t))$ by the construction of the Sylvester matrix.
	By assumption, there exists $(A, B) \in \Aut(\sylvester(2^t))$ with $A_{1,1}= A_{2,2}=1$ and $\{B_{c \bmod 2^t + 1,j \bmod 2^t + 1}\}\in \{-1,1\}$.
	If $\max\{c, j\} \leq 2^t$ or $\min \{c \bmod 2^t+1, j \bmod 2^t+1\} \geq 2^t+1$, then
	$\left(\mathbb{1}_2
	\otimes
	A,
	\mathbb{1}_2
	\otimes
	B\right)$ is an automorphism of $H$ which satisfies the claim.
	Otherwise, 
	$\left(\begin{pmatrix}
		1 & 0\\
		0 & -1	
	\end{pmatrix}
	\otimes
	A,
	\begin{pmatrix}
		0 & 1\\
		1 & 0	
	\end{pmatrix}
	\otimes
	B\right)$ is an automorphism of $H$ with the desired property.
\end{proof}

We now execute our proof strategy again for the Hadamard matrix $\hadTwelve$ of rank~12.
Up to equivalence this matrix is given by
\begin{equation*}
	\hadTwelve = 
	\begin{pmatrix}[r]
		1 &  1 &  1 &  1 &  1 &  1 &  1 &  1 &  1 &  1 &  1 &  1\\
		1 &  1 &  1 & -1 & -1 & -1 & -1 & -1 & -1 &  1 &  1 &  1\\
		1 &  1 &  1 & -1 & -1 & -1 &  1 &  1 &  1 & -1 & -1 & -1\\
		1 & -1 & -1 &  1 & -1 & -1 & -1 &  1 &  1 & -1 &  1 &  1\\
		1 & -1 & -1 & -1 &  1 & -1 &  1 & -1 &  1 &  1 & -1 &  1\\
		1 & -1 & -1 & -1 & -1 &  1 &  1 &  1 & -1 &  1 &  1 & -1\\
		1 & -1 &  1 & -1 &  1 &  1 & -1 &  1 & -1 & -1 & -1 &  1\\
		1 & -1 &  1 &  1 & -1 &  1 & -1 & -1 &  1 &  1 & -1 & -1\\
		1 & -1 &  1 &  1 &  1 & -1 &  1 & -1 & -1 & -1 &  1 & -1\\
		1 &  1 & -1 & -1 &  1 &  1 & -1 & -1 &  1 & -1 &  1 & -1\\
		1 &  1 & -1 &  1 & -1 &  1 &  1 & -1 & -1 & -1 & -1 &  1\\
		1 &  1 & -1 &  1 &  1 & -1 & -1 &  1 & -1 &  1 & -1 & -1
	\end{pmatrix}.
\end{equation*}

\begin{lemma}
	\label{lem:automorphism-had12} \hfill
	\begin{enumerate}[(i)]
		\item \label{lem:automorphism-had12:3rows}
		For each $i \in \{1,2,3\}$, there exists an automorphism $(A, B) \in \Aut(\hadTwelve)$ such that $A_{i,i} = -1$ and $A_{j,j} = 1$ for all $j \in \{1,2,3\}\setminus \{i\}$.
		\item \label{lem:automorphism-had12:2rows1column}
		Let $c$ be a column.
		For each $j \in \omega_c(\hadTwelve)$, there exists an automorphism $(A, B) \in \Aut(\hadTwelve)$ with $A_{1,1} = A_{2,2}=1$ and $B_{c,j}\in \{1,-1\}$.
	\end{enumerate}
\end{lemma}

\begin{proof}
	For notational brevity, we describe an automorphism $(A,B) \in \Aut(\hadTwelve)$ by a row permutation $\sigma^A$, a row inversion set $\mathcal{I}^A$, a column permutation $\sigma^B$, and a column inversion set $\mathcal{I}^B$ such that  for $\Theta \in \{A,B\}$ we have
	\begin{equation*}
		\Theta_{k,l} \coloneqq
		\begin{cases}
			-1, &\text{ if $\sigma^\Theta(l) = k$ and $l \in \mathcal{I}^\Theta$,}\\
			\phantom{-}1,  &\text{ if $\sigma^\Theta(l) = k$ and $l \notin \mathcal{I}^\Theta$,}\\
			\phantom{-}0,  &\text{ otherwise.}
		\end{cases}
	\end{equation*}
	
	For~\eqref{lem:automorphism-had12:3rows}, the table below specifies for each $i \in \{1,2,3\}$ an automorphism $(A^{(i)},B^{(i)})$ satisfying the required conditions.
	
	\begin{longtable}{lllll} 
		$i$ & $\sigma^{A^{(i)}}$ & $\mathcal{I}^{A^{(i)}}$ & $\sigma^{B^{(i)}}$ & $\mathcal{I}^{B^{(i)}}$ \\ \hline
		1 & $(4\:12\:5\:9)(7\:10\:11\:8)$ & $\{1, 6, 7, 8, 10, 11\}$ & $(1\:6)(2\:4\:3\:5)(7\:10\:8\:11)(9\:12)$ & $\{1, \dots, 12\}$\\
		2 & $(4\:10\:5\:8)(7\:9\:11\:12)$ & $\{2, 6, 7, 9, 11, 12\}$ & $(1\:9)(2\:7\:3\:8)(4\:11\:5\:10)(6\:12)$ & $\{ \}$\\
		3 & $(4\:11\:5\:7)(8\:12\:10\:9)$ & $\{3, 6, 8, 9, 10, 12\}$ & $(1\:12)(2\:10\:3\:11)(4\:7\:5\:8)(6\:9)$ & $\{ \}$
	\end{longtable}
	
	For~\eqref{lem:automorphism-had12:2rows1column}, the table below specifies for each $j \in \{1,\dots,12\}$ an automorphism~$(A^{(j)},B^{(j)})$ that maps the first column to the $j$-th column.
	Given $j \in \omega_c(\hadTwelve)$ for some $c \in \{1,\dots,12\}$, the automorphism $({A^{(c)}}^{-1} A^{(j)},{B^{(c)}}^{-1} B^{(j)})$ satisfies the required conditions:
	if $j,c \notin \omega_1(\hadTwelve)$, then we have ${A^{(c)}}^{-1}_{2,2} = A^{(j)}_{2,2} = -1$ while ${A^{(c)}}^{-1}_{1,1} = A^{(j)}_{1,1} = 1$ and ${B^{(c)}}^{-1}_{c,1} = B^{(j)}_{1,j} = 1$ if $j,c \in \omega_1(\hadTwelve)$.
	Thus, we have $({A^{(c)}}^{-1} A^{(j)})_{1,1} = ({A^{(c)}}^{-1} A^{(j)})_{2,2} = ({B^{(c)}}^{-1} B^{(j)})_{c,j} = 1$.
	A similar  argument applies if $j,c \in \omega_c(\hadTwelve)$.
	
	\begin{longtable}{lllll} 
		$j$ & $\sigma^{A^{(j)}}$ & $\mathcal{I}^{A^{(j)}}$ & $\sigma^{B^{(j)}}$ & $\mathcal{I}^{B^{(j)}}$ \\ \hline
		1 & $(3\:12\:11)(4\:9\:6)(5\:7\:8)$ & $\{ \}$ & $(3\:10\:12)(4\:7\:8)(5\:6\:9)$ & $\{ \}$\\
		2 & $(3\:12)(4\:8)(5\:6)(7\:9)$ & $\{4, 5, 6, 7, 8, 9\}$ & $(1\:2)(3\:10)(4\:7)(5\:9)$ & $\{ \}$\\
		3 & $(3\:9\:7\:8)(4\:10\:12\:5)$ & $\{4, 5, 6, 10, 11, 12\}$ & $(1\:3)(2\:11\:12\:10)(4\:8)(5\:6\:9\:7)$ & $\{ \}$\\
		4 & $(3\:10\:5\:6)(4\:8\:9\:11)$ & $\{2, 3, 5, 6, 7, 10\}$ & $(1\:4)(2\:8\:10\:5)(3\:7\:12\:9)(6\:11)$ & $\{ \}$\\
		5 & $(3\:11\:8\:6)(5\:9\:10\:7)$ & $\{2, 3, 4, 6, 8, 11\}$ & $(1\:5)(2\:8\:10\:4)(3\:9)(6\:12\:7\:11)$ & $\{ \}$\\
		6 & $(3\:12\:4\:5)(6\:7\:8\:10)$ & $\{2, 3, 4, 5, 9, 12\}$ & $(1\:6)(2\:7\:12\:4)(3\:9\:11\:8)(5\:10)$ & $\{ \}$\\
		7 & $(3\:11\:9\:6)(4\:8\:12\:7)$ & $\{2, 4, 7, 8, 10, 12\}$ & $(1\:7)(2\:4\:11\:8)(3\:6)(5\:10\:9\:12)$ & $\{ \}$\\
		8 & $(3\:12\:4\:6)(5\:9\:8\:11)$ & $\{2, 5, 8, 9, 10, 11\}$ & $(1\:8)(2\:4\:11\:7)(3\:5\:12\:6)(9\:10)$ & $\{ \}$\\
		9 & $(3\:10\:8\:5)(6\:7\:11\:9)$ & $\{2, 6, 7, 9, 11, 12\}$ & $(1\:9)(2\:6\:10\:7)(3\:5)(4\:12\:8\:11)$ & $\{ \}$\\
		10 & $(3\:11\:4\:9)(5\:12\:8\:6)$ & $\{3, 4, 7, 9, 10, 11\}$ & $(1\:10)(2\:3\:11\:12)(4\:6\:7\:5)(8\:9)$ & $\{ \}$\\
		11 & $(3\:12)(4\:10)(5\:7)(8\:11)$ & $\{3, 5, 7, 8, 11, 12\}$ & $(1\:11)(2\:12)(4\:5)(6\:8)$ & $\{ \}$\\
		12 & $(3\:12)(4\:11)(6\:9)(8\:10)$ & $\{3, 6, 8, 9, 10, 12\}$ & $(1\:12)(2\:11)(6\:8)(7\:9)$ & $\{ \}$
	\end{longtable}
\end{proof}

Having determined which extended Hadamard graphs are~$3$-ultrahomogeneous, we can precisely classify the degree of regularity of all extended Hadamard graphs.

\begin{theorem}
	\label{thm:high-regularity-of-hadamard-graphs}
	Let $H \in \{-1,1\}^{s \times s}$ be a Hadamard matrix and $(G, \colouring)$ be the corresponding extended Hadamard graph.
	\begin{enumerate}[(i)]
		\item \label{itm: hadamard-graphs-are-3tu} $(G, \colouring)$ is \tupleRegular{3}.
		\item \label{itm: 3ultra-hadamard-graphs} $(G, \colouring)$ is \ultraRegular{3} if and only if $H$ is equivalent to a Sylvester matrix or $s=12$.
		\item \label{itm: 4tuple-regular-hadamard-graphs} $(G, \colouring)$ is \tupleRegular{4} if and only if $(G, \colouring) \cong \extendedHadamardGraph(\sylvester(2))$. Moreover, $\extendedHadamardGraph(\sylvester(2))$ is UH.
	\end{enumerate}
\end{theorem}

\begin{proof}
	We first prove~\eqref{itm: hadamard-graphs-are-3tu}.
	Suppose $U \subseteq V(G)$.
	If $U$ is not an independent set, then $\lambda^R(U) = \lambda^B(U) = 0$.
	Therefore, let $U$ be independent with at most three vertices.
	We set $r_U \coloneqq |U \cap R|$ and $b_U \coloneqq |U \cap B|$.
	If $r_U \geq 2$, then $\lambda^R(U) = 0$.
	If $r_U = 1$, then the vertex $u_r \in U\cap R$ has exactly one red neighbour $v_r$ in $(G, \colouring)$.
	Since $U$ is independent and every blue vertex has a neighbour in $\{u_r, v_r\}$, we obtain $\lambda^R(U) = 1$.
	For the remaining cases we claim that
	\[\lambda^R(U) = \frac{s}{2^{b_U-1}}~\text{if $r_U=0$ and $b_U \in \{1,2,3\}$}.\]
	If $b_U \in \{1,2\}$, then the equation above is a direct consequence of Equation~\eqref{eq: hadGraph amply parameters}.
	If $b_U = 3$, then this follows from Equation~\eqref{eq: 3 indep vertices have same nb of nbs}.
	(The $\lambda^B$-values can be obtained by interchanging the roles of the two colours.)

	\medskip
	\noindent
	We prove~\eqref{itm: 3ultra-hadamard-graphs}.
	If $(G, \colouring)$ is \ultraRegular{3},
	then $\mathcal{A}(H)$ is 3-transitive.
	By Theorem~\ref{thm: 3-transitive-hadamard-matrices} and~ Statement~\eqref{eq: equi-matrices-equi-graphs}, we may assume that $H$ is a Sylvester matrix or $\hadTwelve$.
	Suppose $\{v_1, v_2, v_3\}, \{v_1', v_2', v_3'\} \subseteq V(G)$ and there is an isomorphism between $(G,\colouring)[\{v_1, v_2, v_3\}]$ and $(G, \colouring)[\{v_1', v_2', v_3'\}]$ mapping~$v_i$ to~$v'_i$.
	Due to unique neighbourhood within the $2$-cliques, we may assume that  $vv'\notin E(G)$ for all vertices $v,v' \in \{v_1', v_2', v_3'\}$ with $\colouring(v) = \colouring(v')$.
	
	First assume that $\{v_1, v_2, v_3\}, \{v_1', v_2', v_3'\} \subseteq R$ and thus
	\[
	(G,\colouring)[\{v_1, v_2, v_3\}] \cong (G, \colouring)[\{v_1', v_2', v_3'\}] \cong \overline{K_3}.
	\]
	By the 3-transitivity of $\mathcal{A}(H)$
	there exists automorphisms $\varphi, \varphi' \in \Aut((G, \colouring))$ with $\varphi(v_i), \varphi'(v_i') \in \{r_i^+, r_i^-\}$ for all $i \in \{1,2,3\}$.
	Further, Lemma~\ref{lem: aut of sylverster} (respectively Lemma~\ref{lem:automorphism-had12}\eqref{lem:automorphism-had12:3rows}) provides an automorphism~$\psi_j$ for each $j \in \{1,2,3\}$ such that $\psi_j(r_j^+) = r_j^-$ and $\psi_j(r_{j'}^+) = r_{j'}^+$ for all $j' \in \{1,2,3\} \setminus \{j\}$.
	Therefore, for every set $J \subseteq \{1,2,3\}$ there exists an automorphism~$\psi_J$ with $\psi_J(r_j^+) = r_j^-$ and $\psi_J(r_{j'}^+) = r_{j'}^+$ for all $j \in J$.
	This implies that there is a~$J\subseteq \{1,2,3\}$ 
	such that $\varphi \circ \psi_J \circ \varphi'^{-1}$ extends the colour-preserving isomorphism between $(G,\colouring)[\{v_1, v_2, v_3\}]$ and $(G, \colouring)[\{v_1', v_2', v_3'\}]$.
	
	Now assume $\abs{\{v_1, v_2, v_3\}\cap R} = 2$.
	As in the previous case, there exists $\varphi, \varphi' \in \Aut((G, \colouring))$ and $J \subseteq \{1,2\}$ such that  $(\varphi \circ \psi_J)(v_i),=\varphi'(v_i') \in \{r_i^+, r_i^-\}$ 
	for all $i \in \{1,2\}$.
	Note that we might have $(\varphi \circ \psi_J)(v_3) \neq v_3$ and $\varphi'(v'_3) \neq v'_3$.
	So from Lemma~\ref{lem:aut-of-sylvester:2rows1column} (respectively Lemma~\ref{lem:automorphism-had12}\eqref{lem:automorphism-had12:2rows1column}) we obtain an automorphism~$\rho$ mapping $(\varphi \circ \psi_J)(v_3)$ to $\varphi'(v'_3)$ while $(\varphi \circ \psi_J \circ \rho)(v_i) = \varphi'(v'_i)$ for each $i \in \{1,2\}$.
	Hence the automorphism $\varphi \circ \psi_J \circ \rho \circ \varphi'^{-1}$ is an extension of the colour-preserving isomorphism between $(G,\colouring)[\{v_1, v_2, v_3\}]$ and $(G, \colouring)[\{v_1', v_2', v_3'\}]$.
	
	The remaining cases follow by interchanging the roles of the colours.
	Furthermore, by what we just argued, isomorphisms between induced subgraphs with fewer than 3 vertices also extend to automorphisms of $(G,\colouring)$.
	
	\medskip
	\noindent
	It remains to prove~\eqref{itm: 4tuple-regular-hadamard-graphs}.
	If $(G, \colouring)$ is \tupleRegular{4},
	then there is a parameter $\lambda^R$ ($\lambda^B$) such that any four independent blue (red) vertices have $\lambda^R$ ($\lambda^B$) common red (blue) neighbours in $(G, \colouring)$.
	Let $r \in R$ and $b\in B$ be adjacent in $G$.
	Consider the incidence relation $\mathcal{I}$ with point set $p \coloneqq N^R(b)\setminus \{r\}$, block set
	$\{\!\{ \{r' \in p \mid r'b' \in E(G) \} \mid b' \in N^B(r)\setminus \{b\}\}\!\}$.
	Each three points of $\mathcal{I}$ form together with~$r$ an independent 4-set of $G$ and, hence, there exists $\lambda^B-1$ blocks containing all three points.
	Analogously, each three blocks intersect in $\lambda^R-1$ points.
	In particular, $\mathcal{I}$ is a symmetric 3-design.
	We obtain with~\eqref{eq: degenerate symm 3-designs} that
	$|N_G(b) \cap N_G(b')| \in \{0,1,s-1,s\}$ for each block $b'$ of $\mathcal{I}$.
	Since $G$ is an extended Hadamard graph, we have $|N_G(b) \cap N_G(b')| = \nicefrac{s}{2}$.
	Altogether, we have $s \in \{0,2\}$.
	We obtain $s=2$ since there is no Hadamard matrix of rank~0.
	If $H$ is a $2\times 2$-Hadamard matrix, of which there is only one up to equivalence, then $(G, \chi)\cong \extendedHadamardGraph(\sylvester(2))$.
	
	For the ``moreover'' part, assume that $(G, \chi)\cong \extendedHadamardGraph(\sylvester(2))$ (see Figure~\ref{fig: 8cycleWithDiagonals}).
	Fix two non-adjacent vertices $v_1$ and $v_2$ in $V(G)$.
	For each $u \in V(G)\setminus \{v_1, v_2\}$ we call the triple $(\chi(u), \delta_{u \in N_G(v_1)}, \delta_{u \in N_G(v_2)})$ the \emph{type} of $u$, where $\delta_A$ denotes the indicator function of a statement $A$ (that is~$1$ if~$A$ is true and~$0$ otherwise).\\
	\noindent \textbf{Observation:} No two distinct vertices in $V(G)\setminus \{v_1, v_2\}$ are of the same type.
	
	Let now $S, S' \subseteq V(G)$ be vertex sets for which there exists an isomorphism $\varphi$ between the two corresponding induced subgraphs of $(G, \chi)$.
	If $|S| \leq 3$, then it follows with Part~\eqref{itm: 3ultra-hadamard-graphs} that $\varphi$ extends to an automorphism of $(G, \chi)$.
	Therefore, suppose $|S| \geq 4$.
	In particular, $S$ contains two non-adjacent vertices $w_1$ and $w_2$.
	It follows with Part~\eqref{itm: 3ultra-hadamard-graphs} that the restriction of $\varphi$ to $\{w_1, w_2\}$ extends to an automorphism of $(G, \chi)$. It follows by the observation that this automorphism is an extension of $\varphi$.
\end{proof}

\section{Tricoloured graphs}
\label{sec:finite-graphs-multiple-colours}

\begin{theorem}\label{thm:hadamard:is:discon}
	Let $(G, \colouring)$ be a \tupleRegular{3} 3-coloured graph with colour classes $R, B$, and $Y$ such that each of the graphs $G[R]$, $G[B]$, and $G[Y]$ is a disjoint union of complete graphs.
	If $(G, \colouring)[R \cup B] = \extendedHadamardGraph(H)$ for some Hadamard matrix $H \in \{-1, 1\}^{s\times s}$, then $R$ and~$Y$ are homogeneously connected and~$B$ and $Y$ are homogeneously connected.
\end{theorem}
\begin{proof}
	Throughout this proof, the unique equally-coloured neighbour of a vertex $v$ in a Hadamard graph is denoted by $\tilde{v}$.
	
	\medskip
	\noindent
	Suppose that $(G, \chi)[R \cup Y]$ is an extended Hadamard graph.
	If $(G, \chi)[B \cup Y]$ is not an extended Hadamard graph, then by Theorem~\ref{thm: imprimitive graphs two colours main} the sets $B$ and $Y$ are either homogeneously connected or $G[B \cup Y]$ can be obtained from a matching-connected graph $(G', \chi')$ by first applying a blow-up to the $\chi(B)$-coloured vertices and then to the $\chi(Y)$-coloured vertices of $G'$.
	Suppose $y \in Y$ and $b \in B$.
	Observe that $G[\{y,b\}] \cong G[\{\tilde{y}, b\}]$ (independent of whether $B$ and $Y$ are homogeneously connected or $G[B\cup Y]$ is a blow-up).
	With $N^R(b)\cap N^R(\tilde{b}) = \emptyset$ and $\lambda^R(\{y\}) = s$ it follows that $\lambda^R(\{b,y\}) = \lambda^R(\{\tilde{b}, y\}) = \nicefrac{s}{2}$.
	Hence, the $s$-dimensional vector $v$ with $v_i = 1$ if $b_i^{+} \in N(y)$ and $v_i = -1$ otherwise is orthogonal to each of the $s$ columns of $H$, which is a contradiction since these form an orthogonal basis.
	
	If $(G, \chi)[B\cup Y]$ is an extended Hadamard graph, then
	let $y_1, y_2 \in Y$ be non-adjacent vertices.
	Consider the design with points $N^R(y_1)$ and blocks
	$\{\!\{N(b)\cap N^R(y_1) \mid b \in N^B(y_1)\}\!\}$.
	Since $(G, \chi)$ is \tupleRegular{3}, we obtain that this is a symmetric 2-design.
	In particular, relabelling the vertices of $G$ in such a way that $N^R(y_1) = \{r^+_i \mid 1 \leq i \leq s\}$ and $N^B(y_1) = \{c^+_i \mid 1 \leq i \leq s\}$ yields an extended Hadamard graph whose underlying Hadamard matrix is regular (that is, all rows have the same number of positive entries).
	This implies
	\begin{equation} \label{eq: lambda-b-y}
		\text{$\lambda_G^R(\{b,y_1\}) = \nicefrac{1}{2}(s+\sqrt{s})$ for every $b \in N_G^B(y_1)$.}
	\end{equation}
	Let $r \in N^R(y_1) \cap N^R(y_2)$.
	From the 3-tuple regularity of $(G, \colouring)$ we obtain
	\begin{equation*}
		\nicefrac{s}{2} =
		\lambda^B(\{\tilde{y_1}, y_2\}) =  \lambda^B(\{\tilde{y_1},r,y_2\})+\lambda^B(\{\tilde{y_1},\tilde{r},y_2\}) = \lambda^B(\{\tilde{y_1},r,y_2 \}) + \lambda^B(\{y_1,r,\tilde{y_2} \})
	\end{equation*}
	and, hence,
	$\lambda^B(\{y_1, r, \tilde{y_2} \}) = 
	\lambda^B(\{\tilde{y_1}, r, y_2 \}) 	 =  
	\nicefrac{s}{2} - \lambda^B(\{y_1, r, \tilde{y_2} \})$.
	In particular,
	\begin{equation} \label{eq: lambda-y1-r-bary2}
		\lambda^B(\set{y_1, r, \tilde{y_2}}) = \lambda^B(\{\tilde{y_1},r, y_2\}) = \nicefrac{s}{4}.
	\end{equation}
	The equations~\eqref{eq: lambda-b-y} and~\eqref{eq: lambda-y1-r-bary2} imply
	\begin{align*}
		\frac{s}{4}(s + \sqrt{s})
		&= \sum_{b \in N^B(y_1) \cap N^B(\tilde{y_2})}\  \left|N^R(b)\cap N^R(y_1)\right|\\
		&= \sum_{r \in N^R(y_1)}\left|N^B(y_1)\cap N^B(\tilde{y_2}) \cap N^B(r)\right|
		= \frac{s^2}{4}
	\end{align*}
	which is a contradiction.
	
	Therefore, we may assume that neither $G[B \cup Y]$ nor $G[R \cup Y]$ is an extended Hadamard graph.
	If $G[B \cup Y]$ and $G[R \cup Y]$ both arise from blow-ups, then there is a exists a triangle in $G$ formed by vertices $r \in R$, $b \in B$ and $y \in Y$.
	Let $b' \in N^B(r)\setminus \{b\}$.
	Since $G[R\cup B]$ is a Hadamard-graph, the vertices $b$ and $b'$ are non-adjacent.
	This implies that $b' \notin N(y')$ for all $y' \in N^Y(y) \cup \{y\}$.
	Altogether, we have $\lambda^Y(\{r,b\}) = |N^Y(y)| + 1$ and $\lambda^Y(\{r, b'\}) = 0$ which contradicts that $(G, \chi)$ is \tupleRegular{3}.
	
	Hence, we may assume that $G[R \cup Y]$  arises from blow-ups whereas $B$ and $Y$ are homogeneously connected.
	Since $G[R\cup B]$ is an extended Hadamard graph, we know that there are two non-adjacent vertices $b, b' \in B$ and two 2-cliques $C^R_1, C^R_2 \subseteq R$ such that $|C^R_1 \cap N^R(b) \cap N^R(b')| = 1$ and $|C^R_2 \cap N^R(b) \cap N^R(b')| = 0$. 
	Choose $y_i \in Y$ such that $N^R(y_i) = C^R_i$ for $i \in \{1,2\}$.
	We obtain $\lambda^R(\{y_1, b , b'\}) = 0$ and $\lambda^R(\{y_2, b, b'\}) = 2$, which is a contradiction since $(G, \chi)[\{y_1, b , b'\}] \cong (G, \chi)[\{y_2, b , b'\}]$.
\end{proof}

\section{Reductions and classification theorems}\label{sec:classfication:thms}

Towards describing our classification theorem, we observe that there are four simple, generic operations that allow us to create highly regular graphs from other ones.

A graph~$(H, \chi_H)$ is the \emph{colour complementation} of another graph~$(G, \chi_G)$ if~$H$ is obtained from~$G$ by replacing all edges with endpoints in two colour classes~$C$ and~$C'$ (possibly~$C=C')$ with non-edges and vice versa. More formally, 
if~$V(G)=V(H)$ and~$\chi_H=\chi_G$ and there are possibly equal colour classes~$C,C'$ such that~$E(H)= (E(G) \setminus (C\times C'))\cup ((C\times C') \setminus (E(G)\cup C^2\cup {C'}^2))$.

\begin{lemma}
	Suppose~\highlyRegularnok{} $\in \{UH,TR\}$ and~$k\in \mathbb{N}_{\geq 1}$. If~$(H, \chi_H)$ is the colour complementation of~$(G, \chi_G)$, then~$(H, \chi_H)$ is~\highlyRegular{k}{} if and only~$(G, \chi_G)$ is \highlyRegular{k}{}.
\end{lemma}

\begin{proof}
	The lemma follows since the isomorphisms between induced subgraphs of~$(H, \chi_H)$ are precisely the isomorphisms between induced subgraphs of~$(G, \chi_G)$.
\end{proof}

The second operation we need to consider is a homogeneous blow-up.

\begin{lemma}
	Suppose~\highlyRegularnok{} $\in \{UH,TR\}$ and~$k\in \mathbb{N}_{\geq 1}$. If~$(H, \chi_H)$ is a homogeneous blow-up of~$(G, \chi_G)$, then~$(H, \chi_H)$ is~\highlyRegular{k}{} if and only~$(G, \chi_G)$ is \highlyRegular{k}{}.
\end{lemma}

\begin{proof}Assume that $(G, \chi_G)$ has a colour class, and call the vertices within it red.
	Assume further that the red vertices in $(G, \chi_G)$ form an independent set and $(H, \chi_H)$ is obtained by blowing-up the red vertices in~$(G, \chi_G)$ to~$t$-cliques. For a vertex~$w\in V(H)$ we denote by~$\overline{w}$ the vertex of~$V(G)$ from which it originated. (I.e., if~$w$ is red, then it is the red vertex of~$G$ that was blown up to create~$w$. If~$w$ is not red, then~$\overline{w}=w$.)
	
	We call two red vertices of~$(H, \chi_H)$ equivalent if they are adjacent. Vertices of other colours are only equivalent to themselves.
	Every induced subgraph~$X$ of~$(H, \chi_H)$ is associated with an annotated induced subgraph~$\overline{X}$ of~$(G, \chi_G)$ on vertex set~$\{\overline{x}\mid x\in V(X)\}$, where red vertices~$w$ in~$\overline{X}$ are annotated with the number equivalent red vertices of~$H$ they represent.
	
	\noindent \textbf{Observation 1:} Two induced subgraphs of~$(H, \chi_H)$ are isomorphic exactly if their annotated subgraphs of~$(G, \chi_G)$ are isomorphic under an isomorphism that respects annotations. 
	
	Call a bijection between two sets of vertices in~$(H, \chi_H)$ \emph{admissible} if it  maps equivalent vertices to equivalent vertices.
	Note that every isomorphism between two induced subgraphs of~$H$
	is admissible. Every admissible map~$\varphi$ induces a bijection~$\overline{\varphi}$ between vertex sets in~$(G, \chi_G)$.
	An admissible map~$\varphi$ is an isomorphism exactly if~$\overline{\varphi}$ is an isomorphism respecting annotations. 
	
	Note that for every annotated induced subgraph~$Y$ of~$(G, \chi_G)$ whose red annotations are at most~$t$ there is an induced subgraph~$X$ of~$(H, \chi_H)$ so that~$\overline{X}=Y$. 
	
	Regarding ultrahomogeneity, suppose~$(H, \chi_H)$ is  k-$UH$ and suppose there is an isomorphism~$\varphi'$ between two induced subgraphs~$Y_1$ and~$Y_2$ of~$(G, \chi_G)$ of order at most~$k$.
	Consider two induced subgraphs~$X_1$ and~$X_2$ of~$(H, \chi_H)$ whose associated annotated graphs are~$\overline{X_1}=Y_1$ and~$\overline{X_2}=Y_2$ with all annotations of red vertices being~1. 
	It follows that between~$X_1$ and~$X_2$ there is a isomorphism~$\varphi$ with~$\overline{\varphi}=\varphi'$. Since~$(H, \chi_H)$ is k-$UH$, in~$(H, \chi_H)$ there is an automorphism~$\psi$ extending~$\varphi$. This automorphism induces an automorphism~$\overline{\psi}$ of~$(G, \chi_G)$ extending~$\varphi'$.
	Conversely if~$(G, \chi_G)$ is k-$UH$ and there is an isomorphism~$\varphi$ between two induced graphs~$X_1$ and~$X_2$ in~$(H, \chi_H)$, this induces an isomorphism between the associated annotated subgraphs~$\overline{X_1}$ and~$\overline{X_2}$. If this isomorphism extends to an automorphism of~$(G, \chi_G)$, then this automorphism lifts to an automorphism of~$(H, \chi_H)$ extending~$\varphi$.
	
	Regarding tuple regularity, we use the following observation that says that the number of 1-vertex extensions of an induced subgraph~$X$ of a particular isomorphism type can be deduced from the annotated graph.

	\noindent \textbf{Observation 2:} 
	Assume that~$X$ is an induced subgraph of~$(H, \chi_H)$ and let~$\overline{X}$ be the associated annotated subgraph of~$(G, \chi_{G})$. For~$v\in V(G)$, the number of 1-vertex extensions~$X^{+w}$ of~$X$ by a new vertex~$w$ satisfying~$\overline{w}=v$ is equal to
	
	\begin{itemize}
		\item $1$ \tabto{1cm} if~$v$ is not in~$V(\overline{X})$ and not red,
		\item $t$ \tabto{1cm}if~$v$ is not in~$V(\overline{X})$ and \phantom{not} red,
		\item $0$ \tabto{1cm}if~$v$ is \phantom{not} in~$V(\overline{X})$ and not red, and
		\item $t-i$ \tabto{1cm} if~$v$ is \phantom{not}  in~$V(\overline{X})$ and \phantom{not} red with annotation~$i$.
	\end{itemize}
	
	This implies that two isomorphic induced subgraphs in~$(G, \chi_G)$ have the same number of 1-vertex extensions of each isomorphism type if and only if isomorphic lifts of the two graphs have the same number of 1-vertex extensions of each isomorphism type.
\end{proof}

The third operation duplicates all vertices in an independent colour class and connects the duplicates via a matching as follows.
Let $(G,\chi)$ and be $(G', \colouring')$ coloured graphs. A \emph{(possibly) colour-permuting isomorphism} is a bijection $\varphi\colon V(G) \to V(G')$ that is an isomorphism from~$(G, \colouring\circ \pi)$ to $(G', \colouring')$ for some permutation~$\pi$ of the colours~$C$. 

Assume there are distinct colour classes~$R$ and $B$ of $(G,\chi)$ each inducing an independent set for which~$M \coloneqq E(G[R\cup B])$ is a perfect matching of $G[R\cup B]$.
If the bijection from $V(G)\setminus R$ to $V(G) \setminus B$ that is induced by $M$ and fixes all vertices in $V(G)\setminus (R \cup B)$ is a colour-permuting isomorphism, then we say that~$(G,\chi)$ is a \emph{homogeneous matching extension} of~$(G,\chi)[V(G)\setminus R]$ .

\begin{lemma}
	\label{lem:matching}
	Suppose \highlyRegularnok{} $\in \{UH,TR\}$ and~$k\in \mathbb{N}_{\geq 1}$. If~$(H, \colouring_H)$ is a homogeneous matching extension of~$(G, \colouring_G)$, then~$(H, \colouring_H)$ is~\highlyRegular{k}{} if and only~$(G, \colouring_G)$ is \highlyRegular{k}{}.
\end{lemma}

\begin{proof}
	The proof is similar to the previous one. For this note that a homogeneous matching extension is actually a homogeneous blow up with vertices in one colour class replaced be 2-cliques and where the two vertices in each clique obtaining distinct colours.
\end{proof}
For ultrahomogeneity this lemma is also proven in~\cite{truss:homogeneous-coloured-multipartite-graphs}.

Finally, the fourth and last operation combines two graphs with disjoint colours.
We say that~$(H, \chi_H)$ is a \emph{colour disjoint union} of the graphs~$(G, \chi_G)$ and~$(G', \chi_{G'})$ if the colours of~$(G, \chi_G)$ and~$(G', \chi_{G'})$ are disjoint and~$H$ is the disjoint union of~$G$ and~$G'$, where vertices in~$H$ inherit their colour from the respective graph they originate from.

\begin{lemma}
	Suppose~\highlyRegularnok{} $\in \{UH,TR\}$ and~$k\in \mathbb{N}_{\geq 1}$. If~$(H, \chi_H)$ is a colour disjoint union of~$(G, \chi_G)$ and~$(G', \chi_{G'})$, then~$(H, \chi_H)$ is~\highlyRegular{k}{} if and only~$(G, \chi_G)$ and~$(G', \chi_{G'})$ are both \highlyRegular{k}{}. 
\end{lemma}
\begin{proof}
	The lemma follows from the observation that an isomorphism between two induced subgraphs in~$(H, \chi_H)$ naturally decomposes into two isomorphisms between induced subgraphs of~$(G, \chi_G)$ and between induced subgraphs of~$(G', \chi_{G'})$, respectively, and vice versa.
\end{proof}

We call a graph~$(G, \chi)$ \emph{reduced} if neither~$(G, \chi)$ nor any of the graphs obtainable from~$(G, \chi)$ by a sequence of colour complementations is a colour disjoint union, a homogeneous blow-up, or a homogeneous matching extension.
Let us first summarize the bichromatic case using this terminology.
\begin{theorem}\label{thm:4:reg:irreducible:bichrom}
	Suppose~\highlyRegularnok{} $\in \{UH,TR\}$ and~$k \in \mathbb{N}_{\geq 5}$. The only bichromatic reduced~\highlyRegular{k}{} graph is the extended Hadamard graph corresponding to $\sylvester(2)$, whose underlying uncoloured graph is the Wagner graph (Figure~\ref{fig: 8cycleWithDiagonals}).
\end{theorem}
\begin{proof}
	This follows directly by combining the characterization of monochromatic $k$-HR graphs with  Theorems~\ref{thm: 2col primitive graphs main},~\ref{thm: imprimitive graphs two colours main}, and~\ref{thm:high-regularity-of-hadamard-graphs}.
\end{proof}

We can now give a characterization of reduced highly regular graphs.

\begin{theorem}\label{thm:k:reg:irreducible}
	Suppose~\highlyRegularnok{} $\in \{UH,TR\}$ and~$k \in \mathbb{N}_{\geq 5}$. An irreducible graph is~\highlyRegular{k}{} if and only if it is a monochromatic or bichromatic \highlyRegular{k}{} graph.
\end{theorem}
\begin{proof}
	Suppose that~$(G, \chi)$ is~\highlyRegular{k}{} and irreducible. By Lemma~\ref{lem: safeOperations}.\eqref{itm: induced-graph-has-property-pi}, each colour class of~$(G, \chi)$ induces a~\highlyRegular{k}{} graph. 
	If there is a colour class~$R$ in~$(G, \chi)$ that induces a primitive strongly regular graph, then by Lemma~\ref{lem:cricle-components}, Corollary~\ref{lem:rook-graph}, or Lemma~\ref{lem: Sch HS McL do not allow for regular partition} the colour class~$R$ is homogeneously connected to all other colour classes. Since~$(G, \chi)$ is irreducible this implies that~$(G, \chi)$ is monochromatic.
	
	From Theorem~\ref{thm: imprimitive graphs two colours main} it follows that every pair of colour classes that is not homogeneously connected induces an extended Hadamard graph. 
	If~$(G, \chi)$ is not monochromatic, then there are two colour classes that are not homogeneously connected. But then Theorem~\ref{thm:hadamard:is:discon} implies that the graph is bichromatic.
\end{proof}

We have finally assembled all the tools required to state our classification.

\begin{corollary}Suppose~\highlyRegularnok{} $\in \{UH,TR\}$ and~$k \in \mathbb{N}_{\geq 5}$.
	A graph is \highlyRegular{k}{} if and only if it can be obtained from irreducible monochromatic and bichromatic \highlyRegular{k}{} graphs by an arbitrary combination of the following operations:
	\begin{enumerate}
		\item colour disjoint union,
		\item homogeneous matching extension,
		\item homogeneous blow-up, and
		\item colour complementation.
	\end{enumerate}
\end{corollary}

If it turns out that the only monochromatic~$4$-tuple regular graphs that are not~$5$-tuple regular are indeed the McLaughlin graph and its complement then the Theorem and the Corollary also hold for~$k=4$.

When~$k=3$, we can also make a similar statement for graphs whose colour classes induce imprimitive graphs.

\begin{theorem}
	Suppose~\highlyRegularnok{} $\in \{UH,TR\}$. A reduced graph whose colour classes induce imprimitive graphs is~\highlyRegular{3}{} if and only if it is a 1-vertex graph or an extended Hadamard graph.
\end{theorem}
\begin{proof}
	The proof is analogous to the proof of Theorem~\ref{thm:k:reg:irreducible}.
\end{proof}
	
	\section{Conclusion and further research}\label{sec:future:work}
	
	\begin{table}[h!]\renewcommand{\arraystretch}{1.2}
			\centering
			\begin{tabular}{l|c|c|c|c}
				&\!\tupleRegular{3}\!&\!\ultraRegular{3}\!&\!\tupleRegular{4}\!&\!\infUltraRegular{}\!\\
				\cline{1-5}
				\multirow{2}{*}{bichromatic Wagner graph $\extendedHadamardGraph(\sylvester(2))$} &\multirow{2}{*}{$\checkmark$}&\multirow{2}{*}{$\checkmark$}&\multirow{2}{*}{$\checkmark$}&\multirow{2}{*}{$\checkmark$}\\
				 &&&&\\
				 \cline{1-5}
				 				\multirow{2}{*}{possibly graphs made from new monochromatic \tupleRegular{4} graphs} &\multirow{2}{*}{$\checkmark$}&\multirow{2}{*}{$\checkmark$}&\multirow{2}{*}{$\checkmark$}&\\
				 				 &&&&\\
				\cline{1-5}
				
				ext.~Had.~graph of rank 12~Hadamard matrix~$\extendedHadamardGraph(\hadTwelve)$				
				 &\multirow{3}{*}{$\checkmark$}&\multirow{3}{*}{$\checkmark$}&&\\
				ext.~Had.~graph of Sylvester matrix $\extendedHadamardGraph(\sylvester(2^s))$ ($s\geq 2$)&&&&\\
				possibly other graphs with a primitive colour class&&&&\\
				\cline{1-5}
				\multirow{1}{*}{ext.~Had.~graph  $\extendedHadamardGraph(H)$ (not Sylvester matrix and~$H\ncong \hadTwelve$)}		&\multirow{2}{*}{$\checkmark$}&&&\\
				possibly other graphs with a primitive colour class&&&&
				
			\end{tabular}
			\caption{Classification of finite, bichromatic, reduced \tupleRegular{k} and \ultraRegular{\ell} graphs for~$k\geq 5,\ell\geq 4$ and
			classification for finite, bichromatic, reduced \tupleRegular{3}, \tupleRegular{4} and \ultraRegular{3} graphs with both colour classes inducing imprimitive graphs, the McLaughlin or its complement.}
			\label{tab:classifiaction;bichrom}
		\end{table}
		
	In this paper we classified, for~$k\geq 4$ and~$\ell \geq 5$, the~$k$-ultrahomogeneous and the~$\ell$-tuple regular finite graphs. Via several regularity-preserving operations, the classification boils down to analysing bichromatic graphs.
	A two-coloured version of the Wagner graph plays a prominent role in the characterization. 
	For~$k =3$ we showed a classification when the colour classes do not induce primitive graphs. The results are also summarized in Table~\ref{tab:classifiaction;bichrom}.
	
	An open problem that suggests itself is to extend the classification to~$3$-ultrahomogeneous vertex coloured graphs. For this, two things must be done. First one has to analyse whether primitive monochromatic colour classes can be non-homogeneously connected to other colour classes. One cannot rule this out in a straightforward manner as was done in this paper since some of the graphs of interest can actually be partitioned into two parts inducing strongly regular graphs. For example, if $H$ is isomorphic to the Higman-Sims graph and $(V_1,V_2)$ is a partition of $V(H)$ into two strongly regular graphs, then $H[V_1]$ and $H[V_2]$ are both isomorphic to
	the \emph{Hoffman-Singleton graph}
	a strongly regular graph with parameters
	$(50, 7, 0, 1)$. So indeed, a partition into strongly regular graphs is possible.
	If non-trivial bichromatic connections are possible one would have to investigate whether irreducible trichromatic graphs exist.
	
	However, in light of our algorithmic interest in local to global approaches, hinted at in the introduction, a non-obvious avenue of further research might be as follows. Note that both the concepts~$k$-ultrahomogeneity and~$k$-tuple regularity in some way require that among the induced subgraphs of a particular isomorphism type there is only one kind. Indeed, for ultrahomogeneity there should be at most one kind under the orbits of the automorphism group and for tuple regularity there should be at most one kind when considering the multiset of 1-vertex extensions.
	Can we classify or even quantify what happens when we allow several kinds but not too many, say a bounded number of them?
	
	\section*{Acknowledgements}
	We thank Alexander Gavrilyuk for pointing out to us that the classification of monochromatic 4-tuple regulars graphs actually remains an open problem to date and providing pointers to various references.

	\bibliography{bib/inftyTransitiveGraphs}
	
    \newpage
    \appendix
    \noindent {\huge{\textbf{Appendix}}}
    
    \bigskip
    
    \section{Partitioning the Schl\"afli and the McLaughlin graph}
    \label{sec:appendix}
    
    The appendix details the calculations that have to be performed for the proof of Lemma~\ref{lem: Sch HS McL do not allow for regular partition} by hand. We include them for the convenience and overview for reviewers, but do not want to suggest that the tables should be published in the journal.
    We also checked the parameters using the computer. Here we used a brute force algorithm that not require any knowledge on possible parameter sets at all.
    
    Let $(n,d,\lambda,\mu)$ be the parameter set of the Schl\"afli graph (respectively McLaughlin graph).
    As explained in Lemma~\ref{lem: Sch HS McL do not allow for regular partition}, the aim of the calculation is to check whether these two graphs can be partitioned into two strongly regular graphs $H_1$ and $H_2$ with the parameter sets $(n_1,d_1,\lambda_1,\mu_1)$ and $(n_2,d_2,\lambda_2,\mu_2)$, respectively.
    Iterating over Brouwer's list of parameters of strongly regular graphs~\cite{BrouwerParamter:online}, we discard all $(n_1,d_1,\lambda_1,\mu_1)$ which do not satisfy following conditions:
    $n_1 \leq \lfloor\nicefrac{n}{2}\rfloor$, $d_1 \leq d$, $\lambda_1 \leq \lambda$, and  $\mu_1 \leq \mu$.
    For the remaining possible parameter sets, we stepwise compute the parameter set $(n_2,d_2,\lambda_2,\mu_2)$ according to Lemma~\ref{lem: partitionSrgIntoTwoSrgs} in search for a violation. 
    As the tables below display, in each of the cases at least one of the parameters $d_2$, $\lambda_2$, or $\mu_2$ is not a natural number.
    Afterwards, we repeat this process for all parameter sets of the imprimitive strongly regular graphs $sK_t$ and $\overline{tK_{s}}$. Here  we can bound $s$ by the independence number of the graph and $t$ by its clique number.

    \paragraph*{Schl\"afli Graph}
    The parameter set of the Schl\"alfli graph is $(27, 16, 10, 8)$, the clique number is $6$, and the indpendence number is $3$.
    Using Brouwer's parameter list, 27 of 27 parameter combinations were pruned.
    
    \begin{longtable}{c|cccc|cccc|l}
    	& $n_1$ & $d_1$ & $\lambda_1$ & $\mu_1$ & $n_2$ & $d_2$ & $\lambda_2$ & $\mu_2$ & Reason \\ \hline
    	& 5 & 2 & 0 & 1 & 22 & 12.818 &   &   & $d_{2} \notin \Nat$ \\
    	& 9 & 4 & 1 & 2 & 18 & 10 & 5.800 &   & $\lambda_{2} \notin \Nat$ \\
    	& 10 & 3 & 0 & 1 & 17 & 8.353 &   &   & $d_{2} \notin \Nat$ \\
    	& 10 & 6 & 3 & 4 & 17 & 10.118 &   &   & $d_{2} \notin \Nat$ \\
    	& 13 & 6 & 2 & 3 & 14 & 6.714 &   &   & $d_{2} \notin \Nat$ \\
    	\hline$K_1$ & 1 & 0 & - & - & 26 & 15.385 &   &   & $d_{2} \notin \Nat$ \\
    	$K_{2}$ & 2 & 1 & 0 & - & 25 & 14.800 &   &   & $d_{2} \notin \Nat$ \\
    	$K_{3}$ & 3 & 2 & 1 & - & 24 & 14.250 &   &   & $d_{2} \notin \Nat$ \\
    	$K_{4}$ & 4 & 3 & 2 & - & 23 & 13.739 &   &   & $d_{2} \notin \Nat$ \\
    	$K_{5}$ & 5 & 4 & 3 & - & 22 & 13.273 &   &   & $d_{2} \notin \Nat$ \\
    	$K_{6}$ & 6 & 5 & 4 & - & 21 & 12.857 &   &   & $d_{2} \notin \Nat$ \\
    	\hline$\overline{K_{2} }$ & 2 & 0 & - & 0 & 25 & 14.720 &   &   & $d_{2} \notin \Nat$ \\
    	$\overline{K_{3} }$ & 3 & 0 & - & 0 & 24 & 14 & 8.571 &   & $\lambda_{2} \notin \Nat$ \\
    	\hline$2K_{2}$ & 4 & 1 & 0 & 0 & 23 & 13.391 &   &   & $d_{2} \notin \Nat$ \\
    	$3K_{2}$ & 6 & 1 & 0 & 0 & 21 & 11.714 &   &   & $d_{2} \notin \Nat$ \\
    	$2K_{3}$ & 6 & 2 & 1 & 0 & 21 & 12 & 7.095 &   & $\lambda_{2} \notin \Nat$ \\
    	$3K_{3}$ & 9 & 2 & 1 & 0 & 18 & 9 & 3.222 &   & $\lambda_{2} \notin \Nat$ \\
    	$2K_{4}$ & 8 & 3 & 2 & 0 & 19 & 10.526 &   &   & $d_{2} \notin \Nat$ \\
    	$3K_{4}$ & 12 & 3 & 2 & 0 & 15 & 5.600 &   &   & $d_{2} \notin \Nat$ \\
    	$2K_{5}$ & 10 & 4 & 3 & 0 & 17 & 8.941 &   &   & $d_{2} \notin \Nat$ \\
    	$2K_{6}$ & 12 & 5 & 4 & 0 & 15 & 7.200 &   &   & $d_{2} \notin \Nat$ \\
    	\hline$\overline{2K_{2} }$ & 4 & 2 & 0 & 2 & 23 & 13.565 &   &   & $d_{2} \notin \Nat$ \\
    	$\overline{2K_{3} }$ & 6 & 3 & 0 & 3 & 21 & 12.286 &   &   & $d_{2} \notin \Nat$ \\
    	$\overline{3K_{2} }$ & 6 & 4 & 2 & 4 & 21 & 12.571 &   &   & $d_{2} \notin \Nat$ \\
    	$\overline{3K_{3} }$ & 9 & 6 & 3 & 6 & 18 & 11 & 7.364 &   & $\lambda_{2} \notin \Nat$ \\
    	$\overline{4K_{2} }$ & 8 & 6 & 4 & 6 & 19 & 11.789 &   &   & $d_{2} \notin \Nat$ \\
    	$\overline{5K_{2} }$ & 10 & 8 & 6 & 8 & 17 & 11.294 &   &   & $d_{2} \notin \Nat$ \\
    \end{longtable}

    \paragraph*{McLaughlin Graph}
    The parameter set of the Schl\"alfli graph is $(275, 112, 30, 56)$, the clique number is $5$, and the indpendence number is $22$.
    Using Brouwer's parameter list, 372 of 372 parameter combinations were pruned.
    
    \begin{longtable}{c|cccc|cccc|l}
    	& $n_1$ & $d_1$ & $\lambda_1$ & $\mu_1$ & $n_2$ & $d_2$ & $\lambda_2$ & $\mu_2$ & Reason \\ \hline
    	& 5 & 2 & 0 & 1 & 270 & 109.963 &   &   & $d_{2} \notin \Nat$ \\
    	& 9 & 4 & 1 & 2 & 266 & 108.346 &   &   & $d_{2} \notin \Nat$ \\
    	& 10 & 3 & 0 & 1 & 265 & 107.887 &   &   & $d_{2} \notin \Nat$ \\
    	& 10 & 6 & 3 & 4 & 265 & 108 & 28.945 &   & $\lambda_{2} \notin \Nat$ \\
    	& 13 & 6 & 2 & 3 & 262 & 106.740 &   &   & $d_{2} \notin \Nat$ \\
    	& 15 & 6 & 1 & 3 & 260 & 105.885 &   &   & $d_{2} \notin \Nat$ \\
    	& 15 & 8 & 4 & 4 & 260 & 106 & 28.415 &   & $\lambda_{2} \notin \Nat$ \\
    	& 16 & 5 & 0 & 2 & 259 & 105.390 &   &   & $d_{2} \notin \Nat$ \\
    	& 16 & 6 & 2 & 2 & 259 & 105.452 &   &   & $d_{2} \notin \Nat$ \\
    	& 16 & 9 & 4 & 6 & 259 & 105.637 &   &   & $d_{2} \notin \Nat$ \\
    	& 16 & 10 & 6 & 6 & 259 & 105.699 &   &   & $d_{2} \notin \Nat$ \\
    	& 17 & 8 & 3 & 4 & 258 & 105.147 &   &   & $d_{2} \notin \Nat$ \\
    	& 21 & 10 & 3 & 6 & 254 & 103.567 &   &   & $d_{2} \notin \Nat$ \\
    	& 21 & 10 & 4 & 5 & 254 & 103.567 &   &   & $d_{2} \notin \Nat$ \\
    	& 21 & 10 & 5 & 4 & 254 & 103.567 &   &   & $d_{2} \notin \Nat$ \\
    	& 25 & 8 & 3 & 2 & 250 & 101.600 &   &   & $d_{2} \notin \Nat$ \\
    	& 25 & 12 & 5 & 6 & 250 & 102 & 27.353 &   & $\lambda_{2} \notin \Nat$ \\
    	& 25 & 16 & 9 & 12 & 250 & 102.400 &   &   & $d_{2} \notin \Nat$ \\
    	& 26 & 10 & 3 & 4 & 249 & 101.349 &   &   & $d_{2} \notin \Nat$ \\
    	& 26 & 15 & 8 & 9 & 249 & 101.871 &   &   & $d_{2} \notin \Nat$ \\
    	& 27 & 10 & 1 & 5 & 248 & 100.895 &   &   & $d_{2} \notin \Nat$ \\
    	& 27 & 16 & 10 & 8 & 248 & 101.548 &   &   & $d_{2} \notin \Nat$ \\
    	& 28 & 9 & 0 & 4 & 247 & 100.324 &   &   & $d_{2} \notin \Nat$ \\
    	& 28 & 12 & 6 & 4 & 247 & 100.664 &   &   & $d_{2} \notin \Nat$ \\
    	& 28 & 15 & 6 & 10 & 247 & 101.004 &   &   & $d_{2} \notin \Nat$ \\
    	& 28 & 18 & 12 & 10 & 247 & 101.344 &   &   & $d_{2} \notin \Nat$ \\
    	& 29 & 14 & 6 & 7 & 246 & 100.447 &   &   & $d_{2} \notin \Nat$ \\
    	& 33 & 16 & 7 & 8 & 242 & 98.909 &   &   & $d_{2} \notin \Nat$ \\
    	& 35 & 16 & 6 & 8 & 240 & 98 & 26.286 &   & $\lambda_{2} \notin \Nat$ \\
    	& 35 & 18 & 9 & 9 & 240 & 98.292 &   &   & $d_{2} \notin \Nat$ \\
    	& 36 & 10 & 4 & 2 & 239 & 96.636 &   &   & $d_{2} \notin \Nat$ \\
    	& 36 & 14 & 4 & 6 & 239 & 97.238 &   &   & $d_{2} \notin \Nat$ \\
    	& 36 & 14 & 7 & 4 & 239 & 97.238 &   &   & $d_{2} \notin \Nat$ \\
    	& 36 & 15 & 6 & 6 & 239 & 97.389 &   &   & $d_{2} \notin \Nat$ \\
    	& 36 & 20 & 10 & 12 & 239 & 98.142 &   &   & $d_{2} \notin \Nat$ \\
    	& 36 & 21 & 10 & 15 & 239 & 98.293 &   &   & $d_{2} \notin \Nat$ \\
    	& 36 & 21 & 12 & 12 & 239 & 98.293 &   &   & $d_{2} \notin \Nat$ \\
    	& 36 & 25 & 16 & 20 & 239 & 98.895 &   &   & $d_{2} \notin \Nat$ \\
    	& 37 & 18 & 8 & 9 & 238 & 97.387 &   &   & $d_{2} \notin \Nat$ \\
    	& 40 & 12 & 2 & 4 & 235 & 94.979 &   &   & $d_{2} \notin \Nat$ \\
    	& 40 & 27 & 18 & 18 & 235 & 97.532 &   &   & $d_{2} \notin \Nat$ \\
    	& 41 & 20 & 9 & 10 & 234 & 95.880 &   &   & $d_{2} \notin \Nat$ \\
    	& 45 & 12 & 3 & 3 & 230 & 92.435 &   &   & $d_{2} \notin \Nat$ \\
    	& 45 & 16 & 8 & 4 & 230 & 93.217 &   &   & $d_{2} \notin \Nat$ \\
    	& 45 & 22 & 10 & 11 & 230 & 94.391 &   &   & $d_{2} \notin \Nat$ \\
    	& 45 & 28 & 15 & 21 & 230 & 95.565 &   &   & $d_{2} \notin \Nat$ \\
    	& 45 & 32 & 22 & 24 & 230 & 96.348 &   &   & $d_{2} \notin \Nat$ \\
    	& 49 & 12 & 5 & 2 & 226 & 90.319 &   &   & $d_{2} \notin \Nat$ \\
    	& 49 & 16 & 3 & 6 & 226 & 91.186 &   &   & $d_{2} \notin \Nat$ \\
    	& 49 & 18 & 7 & 6 & 226 & 91.619 &   &   & $d_{2} \notin \Nat$ \\
    	& 49 & 24 & 11 & 12 & 226 & 92.920 &   &   & $d_{2} \notin \Nat$ \\
    	& 49 & 30 & 17 & 20 & 226 & 94.221 &   &   & $d_{2} \notin \Nat$ \\
    	& 49 & 32 & 21 & 20 & 226 & 94.655 &   &   & $d_{2} \notin \Nat$ \\
    	& 49 & 36 & 25 & 30 & 226 & 95.522 &   &   & $d_{2} \notin \Nat$ \\
    	& 50 & 7 & 0 & 1 & 225 & 88.667 &   &   & $d_{2} \notin \Nat$ \\
    	& 50 & 21 & 4 & 12 & 225 & 91.778 &   &   & $d_{2} \notin \Nat$ \\
    	& 50 & 21 & 8 & 9 & 225 & 91.778 &   &   & $d_{2} \notin \Nat$ \\
    	& 50 & 28 & 15 & 16 & 225 & 93.333 &   &   & $d_{2} \notin \Nat$ \\
    	& 50 & 28 & 18 & 12 & 225 & 93.333 &   &   & $d_{2} \notin \Nat$ \\
    	& 53 & 26 & 12 & 13 & 222 & 91.468 &   &   & $d_{2} \notin \Nat$ \\
    	& 55 & 18 & 9 & 4 & 220 & 88.500 &   &   & $d_{2} \notin \Nat$ \\
    	& 55 & 36 & 21 & 28 & 220 & 93 & 24.742 &   & $\lambda_{2} \notin \Nat$ \\
    	& 56 & 10 & 0 & 2 & 219 & 85.918 &   &   & $d_{2} \notin \Nat$ \\
    	& 56 & 22 & 3 & 12 & 219 & 88.986 &   &   & $d_{2} \notin \Nat$ \\
    	& 56 & 33 & 22 & 15 & 219 & 91.799 &   &   & $d_{2} \notin \Nat$ \\
    	& 57 & 14 & 1 & 4 & 218 & 86.376 &   &   & $d_{2} \notin \Nat$ \\
    	& 57 & 24 & 11 & 9 & 218 & 88.991 &   &   & $d_{2} \notin \Nat$ \\
    	& 57 & 28 & 13 & 14 & 218 & 90.037 &   &   & $d_{2} \notin \Nat$ \\
    	& 57 & 32 & 16 & 20 & 218 & 91.083 &   &   & $d_{2} \notin \Nat$ \\
    	& 61 & 30 & 14 & 15 & 214 & 88.626 &   &   & $d_{2} \notin \Nat$ \\
    	& 63 & 22 & 1 & 11 & 212 & 85.255 &   &   & $d_{2} \notin \Nat$ \\
    	& 63 & 30 & 13 & 15 & 212 & 87.632 &   &   & $d_{2} \notin \Nat$ \\
    	& 63 & 32 & 16 & 16 & 212 & 88.226 &   &   & $d_{2} \notin \Nat$ \\
    	& 63 & 40 & 28 & 20 & 212 & 90.604 &   &   & $d_{2} \notin \Nat$ \\
    	& 64 & 14 & 6 & 2 & 211 & 82.275 &   &   & $d_{2} \notin \Nat$ \\
    	& 64 & 18 & 2 & 6 & 211 & 83.488 &   &   & $d_{2} \notin \Nat$ \\
    	& 64 & 21 & 0 & 10 & 211 & 84.398 &   &   & $d_{2} \notin \Nat$ \\
    	& 64 & 21 & 8 & 6 & 211 & 84.398 &   &   & $d_{2} \notin \Nat$ \\
    	& 64 & 27 & 10 & 12 & 211 & 86.218 &   &   & $d_{2} \notin \Nat$ \\
    	& 64 & 28 & 12 & 12 & 211 & 86.521 &   &   & $d_{2} \notin \Nat$ \\
    	& 64 & 30 & 18 & 10 & 211 & 87.128 &   &   & $d_{2} \notin \Nat$ \\
    	& 64 & 33 & 12 & 22 & 211 & 88.038 &   &   & $d_{2} \notin \Nat$ \\
    	& 64 & 35 & 18 & 20 & 211 & 88.645 &   &   & $d_{2} \notin \Nat$ \\
    	& 64 & 36 & 20 & 20 & 211 & 88.948 &   &   & $d_{2} \notin \Nat$ \\
    	& 64 & 42 & 26 & 30 & 211 & 90.768 &   &   & $d_{2} \notin \Nat$ \\
    	& 64 & 42 & 30 & 22 & 211 & 90.768 &   &   & $d_{2} \notin \Nat$ \\
    	& 65 & 32 & 15 & 16 & 210 & 87.238 &   &   & $d_{2} \notin \Nat$ \\
    	& 66 & 20 & 10 & 4 & 209 & 82.947 &   &   & $d_{2} \notin \Nat$ \\
    	& 66 & 45 & 28 & 36 & 209 & 90.842 &   &   & $d_{2} \notin \Nat$ \\
    	& 69 & 20 & 7 & 5 & 206 & 81.184 &   &   & $d_{2} \notin \Nat$ \\
    	& 69 & 34 & 16 & 17 & 206 & 85.874 &   &   & $d_{2} \notin \Nat$ \\
    	& 70 & 27 & 12 & 9 & 205 & 82.976 &   &   & $d_{2} \notin \Nat$ \\
    	& 70 & 42 & 23 & 28 & 205 & 88.098 &   &   & $d_{2} \notin \Nat$ \\
    	& 73 & 36 & 17 & 18 & 202 & 84.535 &   &   & $d_{2} \notin \Nat$ \\
    	& 75 & 32 & 10 & 16 & 200 & 82 & 21.951 &   & $\lambda_{2} \notin \Nat$ \\
    	& 75 & 42 & 25 & 21 & 200 & 85.750 &   &   & $d_{2} \notin \Nat$ \\
    	& 76 & 21 & 2 & 7 & 199 & 77.246 &   &   & $d_{2} \notin \Nat$ \\
    	& 76 & 30 & 8 & 14 & 199 & 80.683 &   &   & $d_{2} \notin \Nat$ \\
    	& 76 & 35 & 18 & 14 & 199 & 82.593 &   &   & $d_{2} \notin \Nat$ \\
    	& 76 & 40 & 18 & 24 & 199 & 84.503 &   &   & $d_{2} \notin \Nat$ \\
    	& 76 & 45 & 28 & 24 & 199 & 86.412 &   &   & $d_{2} \notin \Nat$ \\
    	& 77 & 16 & 0 & 4 & 198 & 74.667 &   &   & $d_{2} \notin \Nat$ \\
    	& 77 & 38 & 18 & 19 & 198 & 83.222 &   &   & $d_{2} \notin \Nat$ \\
    	& 78 & 22 & 11 & 4 & 197 & 76.365 &   &   & $d_{2} \notin \Nat$ \\
    	& 81 & 16 & 7 & 2 & 194 & 71.918 &   &   & $d_{2} \notin \Nat$ \\
    	& 81 & 20 & 1 & 6 & 194 & 73.588 &   &   & $d_{2} \notin \Nat$ \\
    	& 81 & 24 & 9 & 6 & 194 & 75.258 &   &   & $d_{2} \notin \Nat$ \\
    	& 81 & 30 & 9 & 12 & 194 & 77.763 &   &   & $d_{2} \notin \Nat$ \\
    	& 81 & 32 & 13 & 12 & 194 & 78.598 &   &   & $d_{2} \notin \Nat$ \\
    	& 81 & 40 & 13 & 26 & 194 & 81.938 &   &   & $d_{2} \notin \Nat$ \\
    	& 81 & 40 & 19 & 20 & 194 & 81.938 &   &   & $d_{2} \notin \Nat$ \\
    	& 81 & 40 & 25 & 14 & 194 & 81.938 &   &   & $d_{2} \notin \Nat$ \\
    	& 81 & 48 & 27 & 30 & 194 & 85.278 &   &   & $d_{2} \notin \Nat$ \\
    	& 82 & 36 & 15 & 16 & 193 & 79.710 &   &   & $d_{2} \notin \Nat$ \\
    	& 82 & 45 & 24 & 25 & 193 & 83.534 &   &   & $d_{2} \notin \Nat$ \\
    	& 85 & 14 & 3 & 2 & 190 & 68.158 &   &   & $d_{2} \notin \Nat$ \\
    	& 85 & 20 & 3 & 5 & 190 & 70.842 &   &   & $d_{2} \notin \Nat$ \\
    	& 85 & 30 & 11 & 10 & 190 & 75.316 &   &   & $d_{2} \notin \Nat$ \\
    	& 85 & 42 & 20 & 21 & 190 & 80.684 &   &   & $d_{2} \notin \Nat$ \\
    	& 88 & 27 & 6 & 9 & 187 & 72 & 17.569 &   & $\lambda_{2} \notin \Nat$ \\
    	& 89 & 44 & 21 & 22 & 186 & 79.462 &   &   & $d_{2} \notin \Nat$ \\
    	& 91 & 24 & 12 & 4 & 184 & 68.478 &   &   & $d_{2} \notin \Nat$ \\
    	& 93 & 46 & 22 & 23 & 182 & 78.275 &   &   & $d_{2} \notin \Nat$ \\
    	& 95 & 40 & 12 & 20 & 180 & 74 & 19.730 &   & $\lambda_{2} \notin \Nat$ \\
    	& 96 & 19 & 2 & 4 & 179 & 62.123 &   &   & $d_{2} \notin \Nat$ \\
    	& 96 & 20 & 4 & 4 & 179 & 62.659 &   &   & $d_{2} \notin \Nat$ \\
    	& 96 & 35 & 10 & 14 & 179 & 70.704 &   &   & $d_{2} \notin \Nat$ \\
    	& 96 & 38 & 10 & 18 & 179 & 72.313 &   &   & $d_{2} \notin \Nat$ \\
    	& 96 & 45 & 24 & 18 & 179 & 76.067 &   &   & $d_{2} \notin \Nat$ \\
    	& 96 & 50 & 22 & 30 & 179 & 78.749 &   &   & $d_{2} \notin \Nat$ \\
    	& 97 & 48 & 23 & 24 & 178 & 77.124 &   &   & $d_{2} \notin \Nat$ \\
    	& 99 & 14 & 1 & 2 & 176 & 56.875 &   &   & $d_{2} \notin \Nat$ \\
    	& 99 & 42 & 21 & 15 & 176 & 72.625 &   &   & $d_{2} \notin \Nat$ \\
    	& 99 & 48 & 22 & 24 & 176 & 76 & 18.632 &   & $\lambda_{2} \notin \Nat$ \\
    	& 99 & 50 & 25 & 25 & 176 & 77.125 &   &   & $d_{2} \notin \Nat$ \\
    	& 99 & 56 & 28 & 36 & 176 & 80.500 &   &   & $d_{2} \notin \Nat$ \\
    	& 100 & 18 & 8 & 2 & 175 & 58.286 &   &   & $d_{2} \notin \Nat$ \\
    	& 100 & 22 & 0 & 6 & 175 & 60.571 &   &   & $d_{2} \notin \Nat$ \\
    	& 100 & 27 & 10 & 6 & 175 & 63.429 &   &   & $d_{2} \notin \Nat$ \\
    	& 100 & 33 & 8 & 12 & 175 & 66.857 &   &   & $d_{2} \notin \Nat$ \\
    	& 100 & 33 & 14 & 9 & 175 & 66.857 &   &   & $d_{2} \notin \Nat$ \\
    	& 100 & 33 & 18 & 7 & 175 & 66.857 &   &   & $d_{2} \notin \Nat$ \\
    	& 100 & 36 & 14 & 12 & 175 & 68.571 &   &   & $d_{2} \notin \Nat$ \\
    	& 100 & 44 & 18 & 20 & 175 & 73.143 &   &   & $d_{2} \notin \Nat$ \\
    	& 100 & 45 & 20 & 20 & 175 & 73.714 &   &   & $d_{2} \notin \Nat$ \\
    	& 100 & 54 & 28 & 30 & 175 & 78.857 &   &   & $d_{2} \notin \Nat$ \\
    	& 100 & 55 & 30 & 30 & 175 & 79.429 &   &   & $d_{2} \notin \Nat$ \\
    	& 101 & 50 & 24 & 25 & 174 & 76.011 &   &   & $d_{2} \notin \Nat$ \\
    	& 105 & 26 & 13 & 4 & 170 & 58.882 &   &   & $d_{2} \notin \Nat$ \\
    	& 105 & 32 & 4 & 12 & 170 & 62.588 &   &   & $d_{2} \notin \Nat$ \\
    	& 105 & 40 & 15 & 15 & 170 & 67.529 &   &   & $d_{2} \notin \Nat$ \\
    	& 105 & 52 & 21 & 30 & 170 & 74.941 &   &   & $d_{2} \notin \Nat$ \\
    	& 105 & 52 & 25 & 26 & 170 & 74.941 &   &   & $d_{2} \notin \Nat$ \\
    	& 105 & 52 & 29 & 22 & 170 & 74.941 &   &   & $d_{2} \notin \Nat$ \\
    	& 109 & 54 & 26 & 27 & 166 & 73.916 &   &   & $d_{2} \notin \Nat$ \\
    	& 111 & 30 & 5 & 9 & 164 & 56.500 &   &   & $d_{2} \notin \Nat$ \\
    	& 111 & 44 & 19 & 16 & 164 & 65.976 &   &   & $d_{2} \notin \Nat$ \\
    	& 112 & 30 & 2 & 10 & 163 & 55.656 &   &   & $d_{2} \notin \Nat$ \\
    	& 112 & 36 & 10 & 12 & 163 & 59.779 &   &   & $d_{2} \notin \Nat$ \\
    	& 113 & 56 & 27 & 28 & 162 & 72.938 &   &   & $d_{2} \notin \Nat$ \\
    	& 115 & 18 & 1 & 3 & 160 & 44.438 &   &   & $d_{2} \notin \Nat$ \\
    	& 117 & 36 & 15 & 9 & 158 & 55.722 &   &   & $d_{2} \notin \Nat$ \\
    	& 117 & 58 & 28 & 29 & 158 & 72.013 &   &   & $d_{2} \notin \Nat$ \\
    	& 119 & 54 & 21 & 27 & 156 & 67.756 &   &   & $d_{2} \notin \Nat$ \\
    	& 120 & 28 & 14 & 4 & 155 & 46.968 &   &   & $d_{2} \notin \Nat$ \\
    	& 120 & 34 & 8 & 10 & 155 & 51.613 &   &   & $d_{2} \notin \Nat$ \\
    	& 120 & 35 & 10 & 10 & 155 & 52.387 &   &   & $d_{2} \notin \Nat$ \\
    	& 120 & 42 & 8 & 18 & 155 & 57.806 &   &   & $d_{2} \notin \Nat$ \\
    	& 120 & 51 & 18 & 24 & 155 & 64.774 &   &   & $d_{2} \notin \Nat$ \\
    	& 120 & 56 & 28 & 24 & 155 & 68.645 &   &   & $d_{2} \notin \Nat$ \\
    	& 120 & 63 & 30 & 36 & 155 & 74.065 &   &   & $d_{2} \notin \Nat$ \\
    	& 121 & 20 & 9 & 2 & 154 & 39.714 &   &   & $d_{2} \notin \Nat$ \\
    	& 121 & 30 & 11 & 6 & 154 & 47.571 &   &   & $d_{2} \notin \Nat$ \\
    	& 121 & 36 & 7 & 12 & 154 & 52.286 &   &   & $d_{2} \notin \Nat$ \\
    	& 121 & 40 & 15 & 12 & 154 & 55.429 &   &   & $d_{2} \notin \Nat$ \\
    	& 121 & 48 & 17 & 20 & 154 & 61.714 &   &   & $d_{2} \notin \Nat$ \\
    	& 121 & 50 & 21 & 20 & 154 & 63.286 &   &   & $d_{2} \notin \Nat$ \\
    	& 121 & 56 & 15 & 35 & 154 & 68 & 20.294 &   & $\lambda_{2} \notin \Nat$ \\
    	& 121 & 60 & 29 & 30 & 154 & 71.143 &   &   & $d_{2} \notin \Nat$ \\
    	& 122 & 55 & 24 & 25 & 153 & 66.549 &   &   & $d_{2} \notin \Nat$ \\
    	& 125 & 28 & 3 & 7 & 150 & 42 & -5 &   & $\lambda_{2} \notin \Nat$ \\
    	& 125 & 48 & 28 & 12 & 150 & 58.667 &   &   & $d_{2} \notin \Nat$ \\
    	& 125 & 52 & 15 & 26 & 150 & 62 & 16.290 &   & $\lambda_{2} \notin \Nat$ \\
    	& 125 & 62 & 30 & 31 & 150 & 70.333 &   &   & $d_{2} \notin \Nat$ \\
    	& 126 & 25 & 8 & 4 & 149 & 38.430 &   &   & $d_{2} \notin \Nat$ \\
    	& 126 & 45 & 12 & 18 & 149 & 55.342 &   &   & $d_{2} \notin \Nat$ \\
    	& 126 & 50 & 13 & 24 & 149 & 59.570 &   &   & $d_{2} \notin \Nat$ \\
    	& 126 & 65 & 28 & 39 & 149 & 72.255 &   &   & $d_{2} \notin \Nat$ \\
    	& 130 & 48 & 20 & 16 & 145 & 54.621 &   &   & $d_{2} \notin \Nat$ \\
    	& 133 & 24 & 5 & 4 & 142 & 29.577 &   &   & $d_{2} \notin \Nat$ \\
    	& 133 & 32 & 6 & 8 & 142 & 37.070 &   &   & $d_{2} \notin \Nat$ \\
    	& 133 & 44 & 15 & 14 & 142 & 48.310 &   &   & $d_{2} \notin \Nat$ \\
    	& 135 & 64 & 28 & 32 & 140 & 65.714 &   &   & $d_{2} \notin \Nat$ \\
    	& 136 & 30 & 8 & 6 & 139 & 31.770 &   &   & $d_{2} \notin \Nat$ \\
    	& 136 & 30 & 15 & 4 & 139 & 31.770 &   &   & $d_{2} \notin \Nat$ \\
    	& 136 & 60 & 24 & 28 & 139 & 61.122 &   &   & $d_{2} \notin \Nat$ \\
    	& 136 & 63 & 30 & 28 & 139 & 64.058 &   &   & $d_{2} \notin \Nat$ \\
    	\hline$K_1$ & 1 & 0 & - & - & 274 & 111.591 &   &   & $d_{2} \notin \Nat$ \\
    	$K_{2}$ & 2 & 1 & 0 & - & 273 & 111.187 &   &   & $d_{2} \notin \Nat$ \\
    	$K_{3}$ & 3 & 2 & 1 & - & 272 & 110.787 &   &   & $d_{2} \notin \Nat$ \\
    	$K_{4}$ & 4 & 3 & 2 & - & 271 & 110.391 &   &   & $d_{2} \notin \Nat$ \\
    	$K_{5}$ & 5 & 4 & 3 & - & 270 & 110 & 29.473 &   & $\lambda_{2} \notin \Nat$ \\
    	\hline$\overline{K_{2} }$ & 2 & 0 & - & 0 & 273 & 111.179 &   &   & $d_{2} \notin \Nat$ \\
    	$\overline{K_{3} }$ & 3 & 0 & - & 0 & 272 & 110.765 &   &   & $d_{2} \notin \Nat$ \\
    	$\overline{K_{4} }$ & 4 & 0 & - & 0 & 271 & 110.347 &   &   & $d_{2} \notin \Nat$ \\
    	$\overline{K_{5} }$ & 5 & 0 & - & 0 & 270 & 109.926 &   &   & $d_{2} \notin \Nat$ \\
    	$\overline{K_{6} }$ & 6 & 0 & - & 0 & 269 & 109.502 &   &   & $d_{2} \notin \Nat$ \\
    	$\overline{K_{7} }$ & 7 & 0 & - & 0 & 268 & 109.075 &   &   & $d_{2} \notin \Nat$ \\
    	$\overline{K_{8} }$ & 8 & 0 & - & 0 & 267 & 108.644 &   &   & $d_{2} \notin \Nat$ \\
    	$\overline{K_{9} }$ & 9 & 0 & - & 0 & 266 & 108.211 &   &   & $d_{2} \notin \Nat$ \\
    	$\overline{K_{10} }$ & 10 & 0 & - & 0 & 265 & 107.774 &   &   & $d_{2} \notin \Nat$ \\
    	$\overline{K_{11} }$ & 11 & 0 & - & 0 & 264 & 107.333 &   &   & $d_{2} \notin \Nat$ \\
    	$\overline{K_{12} }$ & 12 & 0 & - & 0 & 263 & 106.890 &   &   & $d_{2} \notin \Nat$ \\
    	$\overline{K_{13} }$ & 13 & 0 & - & 0 & 262 & 106.443 &   &   & $d_{2} \notin \Nat$ \\
    	$\overline{K_{14} }$ & 14 & 0 & - & 0 & 261 & 105.992 &   &   & $d_{2} \notin \Nat$ \\
    	$\overline{K_{15} }$ & 15 & 0 & - & 0 & 260 & 105.538 &   &   & $d_{2} \notin \Nat$ \\
    	$\overline{K_{16} }$ & 16 & 0 & - & 0 & 259 & 105.081 &   &   & $d_{2} \notin \Nat$ \\
    	$\overline{K_{17} }$ & 17 & 0 & - & 0 & 258 & 104.620 &   &   & $d_{2} \notin \Nat$ \\
    	$\overline{K_{18} }$ & 18 & 0 & - & 0 & 257 & 104.156 &   &   & $d_{2} \notin \Nat$ \\
    	$\overline{K_{19} }$ & 19 & 0 & - & 0 & 256 & 103.688 &   &   & $d_{2} \notin \Nat$ \\
    	$\overline{K_{20} }$ & 20 & 0 & - & 0 & 255 & 103.216 &   &   & $d_{2} \notin \Nat$ \\
    	$\overline{K_{21} }$ & 21 & 0 & - & 0 & 254 & 102.740 &   &   & $d_{2} \notin \Nat$ \\
    	$\overline{K_{22} }$ & 22 & 0 & - & 0 & 253 & 102.261 &   &   & $d_{2} \notin \Nat$ \\
    	\hline$2K_{2}$ & 4 & 1 & 0 & 0 & 271 & 110.362 &   &   & $d_{2} \notin \Nat$ \\
    	$3K_{2}$ & 6 & 1 & 0 & 0 & 269 & 109.524 &   &   & $d_{2} \notin \Nat$ \\
    	$4K_{2}$ & 8 & 1 & 0 & 0 & 267 & 108.674 &   &   & $d_{2} \notin \Nat$ \\
    	$5K_{2}$ & 10 & 1 & 0 & 0 & 265 & 107.811 &   &   & $d_{2} \notin \Nat$ \\
    	$6K_{2}$ & 12 & 1 & 0 & 0 & 263 & 106.935 &   &   & $d_{2} \notin \Nat$ \\
    	$7K_{2}$ & 14 & 1 & 0 & 0 & 261 & 106.046 &   &   & $d_{2} \notin \Nat$ \\
    	$8K_{2}$ & 16 & 1 & 0 & 0 & 259 & 105.143 &   &   & $d_{2} \notin \Nat$ \\
    	$9K_{2}$ & 18 & 1 & 0 & 0 & 257 & 104.226 &   &   & $d_{2} \notin \Nat$ \\
    	$10K_{2}$ & 20 & 1 & 0 & 0 & 255 & 103.294 &   &   & $d_{2} \notin \Nat$ \\
    	$11K_{2}$ & 22 & 1 & 0 & 0 & 253 & 102.348 &   &   & $d_{2} \notin \Nat$ \\
    	$12K_{2}$ & 24 & 1 & 0 & 0 & 251 & 101.386 &   &   & $d_{2} \notin \Nat$ \\
    	$13K_{2}$ & 26 & 1 & 0 & 0 & 249 & 100.410 &   &   & $d_{2} \notin \Nat$ \\
    	$14K_{2}$ & 28 & 1 & 0 & 0 & 247 & 99.417 &   &   & $d_{2} \notin \Nat$ \\
    	$15K_{2}$ & 30 & 1 & 0 & 0 & 245 & 98.408 &   &   & $d_{2} \notin \Nat$ \\
    	$16K_{2}$ & 32 & 1 & 0 & 0 & 243 & 97.383 &   &   & $d_{2} \notin \Nat$ \\
    	$17K_{2}$ & 34 & 1 & 0 & 0 & 241 & 96.340 &   &   & $d_{2} \notin \Nat$ \\
    	$18K_{2}$ & 36 & 1 & 0 & 0 & 239 & 95.280 &   &   & $d_{2} \notin \Nat$ \\
    	$19K_{2}$ & 38 & 1 & 0 & 0 & 237 & 94.203 &   &   & $d_{2} \notin \Nat$ \\
    	$20K_{2}$ & 40 & 1 & 0 & 0 & 235 & 93.106 &   &   & $d_{2} \notin \Nat$ \\
    	$21K_{2}$ & 42 & 1 & 0 & 0 & 233 & 91.991 &   &   & $d_{2} \notin \Nat$ \\
    	$22K_{2}$ & 44 & 1 & 0 & 0 & 231 & 90.857 &   &   & $d_{2} \notin \Nat$ \\
    	$2K_{3}$ & 6 & 2 & 1 & 0 & 269 & 109.546 &   &   & $d_{2} \notin \Nat$ \\
    	$3K_{3}$ & 9 & 2 & 1 & 0 & 266 & 108.278 &   &   & $d_{2} \notin \Nat$ \\
    	$4K_{3}$ & 12 & 2 & 1 & 0 & 263 & 106.981 &   &   & $d_{2} \notin \Nat$ \\
    	$5K_{3}$ & 15 & 2 & 1 & 0 & 260 & 105.654 &   &   & $d_{2} \notin \Nat$ \\
    	$6K_{3}$ & 18 & 2 & 1 & 0 & 257 & 104.296 &   &   & $d_{2} \notin \Nat$ \\
    	$7K_{3}$ & 21 & 2 & 1 & 0 & 254 & 102.906 &   &   & $d_{2} \notin \Nat$ \\
    	$8K_{3}$ & 24 & 2 & 1 & 0 & 251 & 101.482 &   &   & $d_{2} \notin \Nat$ \\
    	$9K_{3}$ & 27 & 2 & 1 & 0 & 248 & 100.024 &   &   & $d_{2} \notin \Nat$ \\
    	$10K_{3}$ & 30 & 2 & 1 & 0 & 245 & 98.531 &   &   & $d_{2} \notin \Nat$ \\
    	$11K_{3}$ & 33 & 2 & 1 & 0 & 242 & 97 & 25.442 &   & $\lambda_{2} \notin \Nat$ \\
    	$12K_{3}$ & 36 & 2 & 1 & 0 & 239 & 95.431 &   &   & $d_{2} \notin \Nat$ \\
    	$13K_{3}$ & 39 & 2 & 1 & 0 & 236 & 93.822 &   &   & $d_{2} \notin \Nat$ \\
    	$14K_{3}$ & 42 & 2 & 1 & 0 & 233 & 92.172 &   &   & $d_{2} \notin \Nat$ \\
    	$15K_{3}$ & 45 & 2 & 1 & 0 & 230 & 90.478 &   &   & $d_{2} \notin \Nat$ \\
    	$16K_{3}$ & 48 & 2 & 1 & 0 & 227 & 88.740 &   &   & $d_{2} \notin \Nat$ \\
    	$17K_{3}$ & 51 & 2 & 1 & 0 & 224 & 86.955 &   &   & $d_{2} \notin \Nat$ \\
    	$18K_{3}$ & 54 & 2 & 1 & 0 & 221 & 85.122 &   &   & $d_{2} \notin \Nat$ \\
    	$19K_{3}$ & 57 & 2 & 1 & 0 & 218 & 83.239 &   &   & $d_{2} \notin \Nat$ \\
    	$20K_{3}$ & 60 & 2 & 1 & 0 & 215 & 81.302 &   &   & $d_{2} \notin \Nat$ \\
    	$21K_{3}$ & 63 & 2 & 1 & 0 & 212 & 79.311 &   &   & $d_{2} \notin \Nat$ \\
    	$22K_{3}$ & 66 & 2 & 1 & 0 & 209 & 77.263 &   &   & $d_{2} \notin \Nat$ \\
    	$2K_{4}$ & 8 & 3 & 2 & 0 & 267 & 108.734 &   &   & $d_{2} \notin \Nat$ \\
    	$3K_{4}$ & 12 & 3 & 2 & 0 & 263 & 107.027 &   &   & $d_{2} \notin \Nat$ \\
    	$4K_{4}$ & 16 & 3 & 2 & 0 & 259 & 105.266 &   &   & $d_{2} \notin \Nat$ \\
    	$5K_{4}$ & 20 & 3 & 2 & 0 & 255 & 103.451 &   &   & $d_{2} \notin \Nat$ \\
    	$6K_{4}$ & 24 & 3 & 2 & 0 & 251 & 101.578 &   &   & $d_{2} \notin \Nat$ \\
    	$7K_{4}$ & 28 & 3 & 2 & 0 & 247 & 99.644 &   &   & $d_{2} \notin \Nat$ \\
    	$8K_{4}$ & 32 & 3 & 2 & 0 & 243 & 97.646 &   &   & $d_{2} \notin \Nat$ \\
    	$9K_{4}$ & 36 & 3 & 2 & 0 & 239 & 95.582 &   &   & $d_{2} \notin \Nat$ \\
    	$10K_{4}$ & 40 & 3 & 2 & 0 & 235 & 93.447 &   &   & $d_{2} \notin \Nat$ \\
    	$11K_{4}$ & 44 & 3 & 2 & 0 & 231 & 91.238 &   &   & $d_{2} \notin \Nat$ \\
    	$12K_{4}$ & 48 & 3 & 2 & 0 & 227 & 88.952 &   &   & $d_{2} \notin \Nat$ \\
    	$13K_{4}$ & 52 & 3 & 2 & 0 & 223 & 86.583 &   &   & $d_{2} \notin \Nat$ \\
    	$14K_{4}$ & 56 & 3 & 2 & 0 & 219 & 84.128 &   &   & $d_{2} \notin \Nat$ \\
    	$15K_{4}$ & 60 & 3 & 2 & 0 & 215 & 81.581 &   &   & $d_{2} \notin \Nat$ \\
    	$16K_{4}$ & 64 & 3 & 2 & 0 & 211 & 78.938 &   &   & $d_{2} \notin \Nat$ \\
    	$17K_{4}$ & 68 & 3 & 2 & 0 & 207 & 76.193 &   &   & $d_{2} \notin \Nat$ \\
    	$18K_{4}$ & 72 & 3 & 2 & 0 & 203 & 73.340 &   &   & $d_{2} \notin \Nat$ \\
    	$19K_{4}$ & 76 & 3 & 2 & 0 & 199 & 70.372 &   &   & $d_{2} \notin \Nat$ \\
    	$20K_{4}$ & 80 & 3 & 2 & 0 & 195 & 67.282 &   &   & $d_{2} \notin \Nat$ \\
    	$21K_{4}$ & 84 & 3 & 2 & 0 & 191 & 64.063 &   &   & $d_{2} \notin \Nat$ \\
    	$22K_{4}$ & 88 & 3 & 2 & 0 & 187 & 60.706 &   &   & $d_{2} \notin \Nat$ \\
    	$2K_{5}$ & 10 & 4 & 3 & 0 & 265 & 107.925 &   &   & $d_{2} \notin \Nat$ \\
    	$3K_{5}$ & 15 & 4 & 3 & 0 & 260 & 105.769 &   &   & $d_{2} \notin \Nat$ \\
    	$4K_{5}$ & 20 & 4 & 3 & 0 & 255 & 103.529 &   &   & $d_{2} \notin \Nat$ \\
    	$5K_{5}$ & 25 & 4 & 3 & 0 & 250 & 101.200 &   &   & $d_{2} \notin \Nat$ \\
    	$6K_{5}$ & 30 & 4 & 3 & 0 & 245 & 98.776 &   &   & $d_{2} \notin \Nat$ \\
    	$7K_{5}$ & 35 & 4 & 3 & 0 & 240 & 96.250 &   &   & $d_{2} \notin \Nat$ \\
    	$8K_{5}$ & 40 & 4 & 3 & 0 & 235 & 93.617 &   &   & $d_{2} \notin \Nat$ \\
    	$9K_{5}$ & 45 & 4 & 3 & 0 & 230 & 90.870 &   &   & $d_{2} \notin \Nat$ \\
    	$10K_{5}$ & 50 & 4 & 3 & 0 & 225 & 88 & 22.091 &   & $\lambda_{2} \notin \Nat$ \\
    	$11K_{5}$ & 55 & 4 & 3 & 0 & 220 & 85 & 20.788 &   & $\lambda_{2} \notin \Nat$ \\
    	$12K_{5}$ & 60 & 4 & 3 & 0 & 215 & 81.860 &   &   & $d_{2} \notin \Nat$ \\
    	$13K_{5}$ & 65 & 4 & 3 & 0 & 210 & 78.571 &   &   & $d_{2} \notin \Nat$ \\
    	$14K_{5}$ & 70 & 4 & 3 & 0 & 205 & 75.122 &   &   & $d_{2} \notin \Nat$ \\
    	$15K_{5}$ & 75 & 4 & 3 & 0 & 200 & 71.500 &   &   & $d_{2} \notin \Nat$ \\
    	$16K_{5}$ & 80 & 4 & 3 & 0 & 195 & 67.692 &   &   & $d_{2} \notin \Nat$ \\
    	$17K_{5}$ & 85 & 4 & 3 & 0 & 190 & 63.684 &   &   & $d_{2} \notin \Nat$ \\
    	$18K_{5}$ & 90 & 4 & 3 & 0 & 185 & 59.459 &   &   & $d_{2} \notin \Nat$ \\
    	$19K_{5}$ & 95 & 4 & 3 & 0 & 180 & 55 & -0.055 &   & $\lambda_{2} \notin \Nat$ \\
    	$20K_{5}$ & 100 & 4 & 3 & 0 & 175 & 50.286 &   &   & $d_{2} \notin \Nat$ \\
    	$21K_{5}$ & 105 & 4 & 3 & 0 & 170 & 45.294 &   &   & $d_{2} \notin \Nat$ \\
    	$22K_{5}$ & 110 & 4 & 3 & 0 & 165 & 40 & -22.200 &   & $\lambda_{2} \notin \Nat$ \\
    	\hline$\overline{2K_{2} }$ & 4 & 2 & 0 & 2 & 271 & 110.376 &   &   & $d_{2} \notin \Nat$ \\
    	$\overline{2K_{3} }$ & 6 & 3 & 0 & 3 & 269 & 109.569 &   &   & $d_{2} \notin \Nat$ \\
    	$\overline{2K_{4} }$ & 8 & 4 & 0 & 4 & 267 & 108.764 &   &   & $d_{2} \notin \Nat$ \\
    	$\overline{2K_{5} }$ & 10 & 5 & 0 & 5 & 265 & 107.962 &   &   & $d_{2} \notin \Nat$ \\
    	$\overline{2K_{6} }$ & 12 & 6 & 0 & 6 & 263 & 107.163 &   &   & $d_{2} \notin \Nat$ \\
    	$\overline{2K_{7} }$ & 14 & 7 & 0 & 7 & 261 & 106.368 &   &   & $d_{2} \notin \Nat$ \\
    	$\overline{2K_{8} }$ & 16 & 8 & 0 & 8 & 259 & 105.575 &   &   & $d_{2} \notin \Nat$ \\
    	$\overline{2K_{9} }$ & 18 & 9 & 0 & 9 & 257 & 104.786 &   &   & $d_{2} \notin \Nat$ \\
    	$\overline{2K_{10} }$ & 20 & 10 & 0 & 10 & 255 & 104 & 27.919 &   & $\lambda_{2} \notin \Nat$ \\
    	$\overline{2K_{11} }$ & 22 & 11 & 0 & 11 & 253 & 103.217 &   &   & $d_{2} \notin \Nat$ \\
    	$\overline{2K_{12} }$ & 24 & 12 & 0 & 12 & 251 & 102.438 &   &   & $d_{2} \notin \Nat$ \\
    	$\overline{2K_{13} }$ & 26 & 13 & 0 & 13 & 249 & 101.663 &   &   & $d_{2} \notin \Nat$ \\
    	$\overline{2K_{14} }$ & 28 & 14 & 0 & 14 & 247 & 100.891 &   &   & $d_{2} \notin \Nat$ \\
    	$\overline{2K_{15} }$ & 30 & 15 & 0 & 15 & 245 & 100.122 &   &   & $d_{2} \notin \Nat$ \\
    	$\overline{2K_{16} }$ & 32 & 16 & 0 & 16 & 243 & 99.358 &   &   & $d_{2} \notin \Nat$ \\
    	$\overline{2K_{17} }$ & 34 & 17 & 0 & 17 & 241 & 98.598 &   &   & $d_{2} \notin \Nat$ \\
    	$\overline{2K_{18} }$ & 36 & 18 & 0 & 18 & 239 & 97.841 &   &   & $d_{2} \notin \Nat$ \\
    	$\overline{2K_{19} }$ & 38 & 19 & 0 & 19 & 237 & 97.089 &   &   & $d_{2} \notin \Nat$ \\
    	$\overline{2K_{20} }$ & 40 & 20 & 0 & 20 & 235 & 96.340 &   &   & $d_{2} \notin \Nat$ \\
    	$\overline{2K_{21} }$ & 42 & 21 & 0 & 21 & 233 & 95.597 &   &   & $d_{2} \notin \Nat$ \\
    	$\overline{2K_{22} }$ & 44 & 22 & 0 & 22 & 231 & 94.857 &   &   & $d_{2} \notin \Nat$ \\
    	$\overline{3K_{2} }$ & 6 & 4 & 2 & 4 & 269 & 109.591 &   &   & $d_{2} \notin \Nat$ \\
    	$\overline{3K_{3} }$ & 9 & 6 & 3 & 6 & 266 & 108.414 &   &   & $d_{2} \notin \Nat$ \\
    	$\overline{3K_{4} }$ & 12 & 8 & 4 & 8 & 263 & 107.255 &   &   & $d_{2} \notin \Nat$ \\
    	$\overline{3K_{5} }$ & 15 & 10 & 5 & 10 & 260 & 106.115 &   &   & $d_{2} \notin \Nat$ \\
    	$\overline{3K_{6} }$ & 18 & 12 & 6 & 12 & 257 & 104.996 &   &   & $d_{2} \notin \Nat$ \\
    	$\overline{3K_{7} }$ & 21 & 14 & 7 & 14 & 254 & 103.898 &   &   & $d_{2} \notin \Nat$ \\
    	$\overline{3K_{8} }$ & 24 & 16 & 8 & 16 & 251 & 102.821 &   &   & $d_{2} \notin \Nat$ \\
    	$\overline{3K_{9} }$ & 27 & 18 & 9 & 18 & 248 & 101.766 &   &   & $d_{2} \notin \Nat$ \\
    	$\overline{3K_{10} }$ & 30 & 20 & 10 & 20 & 245 & 100.735 &   &   & $d_{2} \notin \Nat$ \\
    	$\overline{3K_{11} }$ & 33 & 22 & 11 & 22 & 242 & 99.727 &   &   & $d_{2} \notin \Nat$ \\
    	$\overline{3K_{12} }$ & 36 & 24 & 12 & 24 & 239 & 98.745 &   &   & $d_{2} \notin \Nat$ \\
    	$\overline{3K_{13} }$ & 39 & 26 & 13 & 26 & 236 & 97.788 &   &   & $d_{2} \notin \Nat$ \\
    	$\overline{3K_{14} }$ & 42 & 28 & 14 & 28 & 233 & 96.858 &   &   & $d_{2} \notin \Nat$ \\
    	$\overline{3K_{15} }$ & 45 & 30 & 15 & 30 & 230 & 95.957 &   &   & $d_{2} \notin \Nat$ \\
    	$\overline{3K_{16} }$ & 48 & 32 & 16 & 32 & 227 & 95.084 &   &   & $d_{2} \notin \Nat$ \\
    	$\overline{3K_{17} }$ & 51 & 34 & 17 & 34 & 224 & 94.241 &   &   & $d_{2} \notin \Nat$ \\
    	$\overline{3K_{18} }$ & 54 & 36 & 18 & 36 & 221 & 93.430 &   &   & $d_{2} \notin \Nat$ \\
    	$\overline{3K_{19} }$ & 57 & 38 & 19 & 38 & 218 & 92.651 &   &   & $d_{2} \notin \Nat$ \\
    	$\overline{3K_{20} }$ & 60 & 40 & 20 & 40 & 215 & 91.907 &   &   & $d_{2} \notin \Nat$ \\
    	$\overline{3K_{21} }$ & 63 & 42 & 21 & 42 & 212 & 91.198 &   &   & $d_{2} \notin \Nat$ \\
    	$\overline{3K_{22} }$ & 66 & 44 & 22 & 44 & 209 & 90.526 &   &   & $d_{2} \notin \Nat$ \\
    	$\overline{4K_{2} }$ & 8 & 6 & 4 & 6 & 267 & 108.824 &   &   & $d_{2} \notin \Nat$ \\
    	$\overline{4K_{3} }$ & 12 & 9 & 6 & 9 & 263 & 107.300 &   &   & $d_{2} \notin \Nat$ \\
    	$\overline{4K_{4} }$ & 16 & 12 & 8 & 12 & 259 & 105.822 &   &   & $d_{2} \notin \Nat$ \\
    	$\overline{4K_{5} }$ & 20 & 15 & 10 & 15 & 255 & 104.392 &   &   & $d_{2} \notin \Nat$ \\
    	$\overline{4K_{6} }$ & 24 & 18 & 12 & 18 & 251 & 103.012 &   &   & $d_{2} \notin \Nat$ \\
    	$\overline{4K_{7} }$ & 28 & 21 & 14 & 21 & 247 & 101.684 &   &   & $d_{2} \notin \Nat$ \\
    	$\overline{4K_{8} }$ & 32 & 24 & 16 & 24 & 243 & 100.412 &   &   & $d_{2} \notin \Nat$ \\
    	$\overline{4K_{9} }$ & 36 & 27 & 18 & 27 & 239 & 99.197 &   &   & $d_{2} \notin \Nat$ \\
    	$\overline{4K_{10} }$ & 40 & 30 & 20 & 30 & 235 & 98.043 &   &   & $d_{2} \notin \Nat$ \\
    	$\overline{4K_{11} }$ & 44 & 33 & 22 & 33 & 231 & 96.952 &   &   & $d_{2} \notin \Nat$ \\
    	$\overline{4K_{12} }$ & 48 & 36 & 24 & 36 & 227 & 95.930 &   &   & $d_{2} \notin \Nat$ \\
    	$\overline{4K_{13} }$ & 52 & 39 & 26 & 39 & 223 & 94.978 &   &   & $d_{2} \notin \Nat$ \\
    	$\overline{4K_{14} }$ & 56 & 42 & 28 & 42 & 219 & 94.100 &   &   & $d_{2} \notin \Nat$ \\
    	$\overline{4K_{15} }$ & 60 & 45 & 30 & 45 & 215 & 93.302 &   &   & $d_{2} \notin \Nat$ \\
    	$\overline{5K_{2} }$ & 10 & 8 & 6 & 8 & 265 & 108.075 &   &   & $d_{2} \notin \Nat$ \\
    	$\overline{5K_{3} }$ & 15 & 12 & 9 & 12 & 260 & 106.231 &   &   & $d_{2} \notin \Nat$ \\
    	$\overline{5K_{4} }$ & 20 & 16 & 12 & 16 & 255 & 104.471 &   &   & $d_{2} \notin \Nat$ \\
    	$\overline{5K_{5} }$ & 25 & 20 & 15 & 20 & 250 & 102.800 &   &   & $d_{2} \notin \Nat$ \\
    	$\overline{5K_{6} }$ & 30 & 24 & 18 & 24 & 245 & 101.224 &   &   & $d_{2} \notin \Nat$ \\
    	$\overline{5K_{7} }$ & 35 & 28 & 21 & 28 & 240 & 99.750 &   &   & $d_{2} \notin \Nat$ \\
    	$\overline{5K_{8} }$ & 40 & 32 & 24 & 32 & 235 & 98.383 &   &   & $d_{2} \notin \Nat$ \\
    	$\overline{5K_{9} }$ & 45 & 36 & 27 & 36 & 230 & 97.130 &   &   & $d_{2} \notin \Nat$ \\
    	$\overline{5K_{10} }$ & 50 & 40 & 30 & 40 & 225 & 96 & 25 & 52.500 & $\mu_{2} \notin \Nat$ \\
    \end{longtable}

\end{document}